\RequirePackage{ifpdf}
\ifpdf 
\documentclass[pdftex]{IPA}
\else
\documentclass{IPA}
\fi

\usepackage{amsmath}
\usepackage{amsthm} 
\usepackage{amsfonts} 
\usepackage{amssymb} 
\usepackage{graphicx}
\usepackage{appendix}
\usepackage{hyperref}
\usepackage{bbm}
\usepackage{mathrsfs}
\usepackage{pgf,tikz}
\usepackage{mathrsfs}
\usepackage{stmaryrd}
\usepackage{enumitem}
\usepackage[all]{xy}
\usetikzlibrary{arrows}

\DeclareMathAlphabet{\mathpzc}{OT1}{pzc}{m}{it}

\renewcommand{\hat}{\widehat}

\newtheorem{teo}{Theorem}[section]

\newtheorem{claim}[teo]{Claim}

\numberwithin{equation}{section}

\begin{document}
\thispagestyle{empty}

\ShortArticleName{On the splitting of infinitesimal Poisson automorphisms around symplectic leaves}

\ArticleName{On the splitting of infinitesimal Poisson automorphisms around symplectic leaves}

\Author{Eduardo VELASCO-BARRERAS~$^\dag$ and Yury VOROBIEV~$^\ddag$} \AuthorNameForHeading{E. Velasco-Barreras and Yu. Vorobiev}

\Address{Department of Mathematics, University of Sonora\\
Rosales y Blvd. Luis Encinas, 83000 Hermosillo, M\'exico}

\Email{\href{mailto:lalo.velasco@mat.uson.mx}{lalo.velasco@mat.uson.mx}$^\dag$, \href{mailto:yurimv@guaymas.uson.mx}{yurimv@guaymas.uson.mx}$^\ddag$}

\sloppy 

\Abstract{A geometric description of the first Poisson cohomology groups is given in the semilocal context, around (possibly singular)
symplectic leaves. This result is based on the splitting theorems for infinitesimal automorphisms of coupling Poisson structures
which describe the interaction between the tangential and transversal data of the characteristic distributions. As a
consequence, we derive some criteria of vanishing of the first Poisson cohomology groups and apply the general splitting formulas
to some particular classes of Poisson structures associated with singular symplectic foliations.
}


\Classification{53D17; 53C05; 53C12}

\section{Introduction}

This paper is devoted to the study of infinitesimal automorphisms of Poisson manifolds carrying singular symplectic foliations.

Let $(M,\Psi)$ be a Poisson manifold equipped with a Poisson tensor $\Psi \in\mathfrak{X}^{2}(M)$, that is, a bivector field on $M$ satisfying the Jacobi identity $[\Psi,\Psi]=0$. Here $[~,~]$ denotes the Schouten-Nijenhuis bracket for multivector fields on $M$ which induces the Lichnerowicz coboundary operator $\delta_{\Psi}:\mathfrak{X}^{p}(M)\rightarrow\mathfrak{X}^{p+1}(M)$ given by $\delta_{\Psi}(\cdot)=[\Psi,\cdot]$. Recall that the Poisson cohomology of $(M,\Psi)$ is just the cohomology of the cochain complex $(\mathfrak{X}^{\ast}(M),\delta_{\Psi})$ whose spaces are denoted by $H_{\Psi}^{p}(M)$. In particular, $H_{\Psi}^{0}(M)=\operatorname{Casim}(M,\Psi)$ is the space of Casimir functions and the first cohomology group $H_{\Psi}^{1}(M)$ \ is the quotient of the Lie algebra of the infinitesimal Poisson automorphisms (Poisson vector fields)
\[
\operatorname{Poiss}(M,\Psi)=\{Z\in\mathfrak{X}(M)\mid L_{Z}\Psi\equiv[Z,\Psi]=0\}
\]
with respect to the ideal of the Hamiltonian vector fields
\[
\operatorname{Ham}(M,\Psi)=\{X_{F}=\Psi^{\sharp}dF\mid F\in C^{\infty}(M)\}.
\]
Here, $\Psi^{\sharp}:T^{\ast}M\rightarrow TM$ is the induced vector bundle morphism defined by $\langle\beta,\Psi^{\sharp}\alpha\rangle=\Psi(\alpha,\beta)$.

The Hamiltonian vector fields generate the generalized characteristic distribution $\Psi^{\sharp}(T^{\ast}M)\subseteq TM$ which is integrable in the sense of Stefan-Sussmann and gives rise to a singular symplectic foliation $(\mathcal{S},\omega)$ equipped with a leafwise symplectic form $\omega$. The singular situation occurs when the rank of the Poisson tensor $\Psi$ is not locally constant at some points in $M$. In this case, the leafwise symplectic form $\omega$ has a singular behavior, in the sense that it cannot be represented as the pull-back of a global 2-form on $M$.

In the regular case, when $\operatorname{rank}\Psi=\mathrm{const}$ on $M$, one can associate to the regular symplectic foliation $\mathcal{S}$ the cochain complex $(\Gamma(\wedge^{\ast}T^{\ast}\mathcal{S)},d_{\mathcal{S}})$, where $d_{\mathcal{S}}$ is the \emph{foliated exterior differential} for foliated forms. This gives rise to the foliated (or, tangential) de Rham cohomology groups $H_{\operatorname{dR}}^{k}(\mathcal{S})$. Fixing a normal subbundle $\nu(\mathcal{S})\subset TM$ of the regular symplectic foliation $\mathcal{S}$,
\[
TM=T\mathcal{S}\oplus\nu(\mathcal{S}),
\]
we can introduce the subspace $\Gamma_{\mathcal{S}\text{-}\mathrm{pr}}(\nu(\mathcal{S}))$ of all smooth sections $Y$ of $\nu(\mathcal{S})$ which are projectable to the leaf space $M\diagup\mathcal{S}$, that is, $[Y,\Gamma(T\mathcal{S})] \subset \Gamma(T\mathcal{S})$. Then, for the first cohomology group of the regular Poisson manifold, we have the following splitting formula into tangential and transversal components \cite{Va-90,VK-88}:
\begin{equation}\label{F1}
H_{\Psi}^{1}(M) \cong H_{\mathrm{dR}}^{1}(\mathcal{S})\oplus\left\{Y\in\Gamma_{\mathcal{S}\text{-}\mathrm{pr}}(\nu(\mathcal{S}))\mid L_{Y}\omega\text{ is }d_{\mathcal{S}}\text{-exact}\right\}.
\end{equation}
This formula involves the Lie derivative $L_{Y}\omega$ of the leafwise symplectic form along a projectable vector field on $(M,\mathcal{S})$ which is a well-defined $d_{\mathcal{S}}$-closed, foliated 2-form, in general. The tangential component in \eqref{F1} is independent of $\omega$ and just coincides with the first foliated de Rham cohomology group of $\mathcal{S}$. Notice that the natural homomorphism $H^1_{\mathrm{dR}}(M) \rightarrow H^1_{\mathrm{dR}}(\mathcal{S})$ induced by the canonical inclusion of the $\mathcal{S}$-leaves is not surjective in general, and hence the image under the homomorphism $H_{\mathrm{dR}}^{1}(M)\rightarrow H_{\Psi}^{1}(M)$ induced by $\Psi^{\sharp}$ is not necessarily isomorphic to the tangential component $H^1_{\mathrm{dR}}(\mathcal{S})$. The transversal component in \eqref{F1} consists of the $\nu(\mathcal{S})$-valued infinitesimal automorphisms $Y$ of the symplectic foliation which correspond to the cohomologically trivial in $H_{\operatorname{dR}}^{2}(\mathcal{S})$ transversal variations $L_{Y}\omega$ of $\omega$. Notice that this component coincides also with the image under the homomorphism $H_{\Psi}^{1}(M)\rightarrow\Gamma(\nu(\mathcal{S}))$ induced by the projection $TM\rightarrow\nu(\mathcal{S})$ along $T\mathcal{S}$ and hence its non-triviality implies the existence of non-vanishing Poisson vector fields of $\Psi$ transversal to the symplectic foliation. Formula \eqref{F1} can be used to compute the first Poisson cohomology in some particular cases \cite{AG,KaMaslov,Va-94,Va-90,VK-88,Xu-92}.

Our purpose is to understand how far the above results can be generalized in the singular case. We show that there is the following analog of \eqref{F1} in the semilocal context, around a possibly singular symplectic leaf.

\begin{claim}\label{Claim:Semilocal}
Let $S\subset M$ be an embedded symplectic leaf of the Poisson manifold $(M,\Psi)$. Then, there exists a tubular neighborhood $N$ of $S$ in $M$ such that
\begin{equation}\label{S1}
\Psi=\Psi_{H}+\Psi_{V}\text{ on }N,
\end{equation}
where $\Psi_{H}$ is a bivector field on $N$ of constant rank, $\operatorname{rank}\Psi_{H}=\dim S$ and $\Psi_{V}$ is a Poisson bivector field on $N$ vanishing at $S$. Moreover, the first cohomology group of the Poisson structure $\Psi$ admits the following splitting:
\begin{equation}\label{S2}
H_{\Psi}^{1}(N)\cong H_{\bar{\partial}}^{1}\oplus\frac{\ker\{\rho:\mathcal{A}\longrightarrow H_{\bar{\partial}}^{2}\}}{\operatorname{Ham}(N,\Psi_{V})}.
\end{equation}
\end{claim}

The first part of Claim \ref{Claim:Semilocal} is just the known fact \cite{Vo-01} on the existence of a coupling neighborhood of a closed symplectic leaf. One can think of the bivector fields $\Psi_{H}$ and $\Psi_{V}$ in decomposition \eqref{S1} as the ``regular" and ``singular" parts of $\Psi$, respectively. The horizontal subbundle $\mathbb{H}:=\Psi_{H}^{\sharp }(T^{\ast}N)\subset TN$ is a regular (possibly nonintegrable) distribution on $N$ of rank $\mathbb{H}=\dim S$ which belongs to $T\mathcal{S}$. The Poisson bivector field $\Psi_{V}$ (of varying rank, in general) is said to be a \emph{transverse Poisson structure} around the leaf $S$ and characterized by the property: the characteristic distribution of $\Psi_{V}$ belongs to the vertical subbundle of the tubular neighborhood. The restriction of $\Psi_{V}$ to each fiber is just the local transverse Poisson structure whose existence and uniqueness are provided by Weinstein's splitting theorem \cite{We-83}.

The first ``tangential" term $H_{\bar{\partial}}^{\ast}$ in the splitting \eqref{S2} is the cohomology of the cochain complex  $(\mathcal{C}^{\ast}=\oplus_{p}\mathcal{C}^{p},\bar{\partial})$, where $\mathcal{C}^{p}:=\Omega^{p}(S)\otimes_{C^{\infty}(S)}\operatorname{Casim}(N,\Psi_{V})$ is the $C^{\infty}(S)$-module of $p$-forms on $S$ with values in Casimir functions of $\Psi_{V}$. The coboundary operator $\bar{\partial}$ is defined in terms of the covariant exterior derivative associated with $\mathbb{H}$. The second, ``transversal" term in \eqref{S2} involves the kernel of an intrinsic morphism\ $\rho$ from a Lie subalgebra $\mathcal{A}$ of $\operatorname{Poiss}(N,\Psi_{V})$ to the second cohomology group $H_{\bar{\partial}}^{2}$ of $\bar{\partial}$. Notice that, in the case when the symplectic leaf $S$ is regular, the tubular neighborhood $N$ can be chosen in such a way that $\Psi_{V}\equiv0$. Then, $\mathbb{H}\equiv TS$ and formula \eqref{S2} coincides with \eqref{F1} under an appropriate choice of $\nu(\mathcal{S})$.

We give a geometric derivation of formula \eqref{S2} which is related to the description of infinitesimal automorphisms of coupling Poisson structures \cite{Vo-01} and based on the Schouten-Ehresmann bigraded calculus on fibered and foliated manifolds \cite{Va-04,Vo-05}. Algebraically, the study of the Poisson cohomology around $S$ deals with a class of bigraded cochain complexes appearing in various contexts, for example, in \cite{Bra-10,CrFe-10,CrMa-13,ItskovYuraKarasev,Marcut-13,Va-94}. Using some natural filtrations of these complexes, we show that formula \eqref{S2} can be also derived in the framework of the theory of spectral sequences. As well as in the regular case \cite{Va-90}, the main motivation of this algebraic approach is related to further computations of the Poisson cohomology in higher degrees.

We apply formula \eqref{S2} to two particular cases which are related with singular symplectic foliations. Firstly, by using Conn's result \cite{Conn-85}, we formulate some sufficient criteria for the triviality of $H_{\Psi}^{1}(N)$, in the case when the isotropy algebra of the symplectic leaf $S$ is a semisimple Lie algebra of compact type. The main point here is to regard the cochain complex $(\mathcal{C}^{*},\bar{\partial})$ as a subcomplex of a foliated de Rham complex. This condition can be realized when the symplectic leaf $S$ is ``flat" in the following sense: there exists a regular foliation $\mathcal{F}$ on $N$ for which $S$ is a leaf and such that the projectable sections of $T\mathcal{F}$ are infinitesimal automorphisms of $\Psi_{V}$. Then, $(\mathcal{C}^{*},\bar{\partial})$ can be viewed as a subcomplex of $(\Gamma(\wedge^{\ast}T^{*}\mathcal{F}),d_{\mathcal{F}})$ and as a consequence we have the natural homomorphism $H_{\bar{\partial}}^{1}\rightarrow H_{\operatorname{dR}}^{1}(\mathcal{F})$. If this homomorphism is injective, then the triviality of the first foliated de Rham cohomology implies the triviality of $H_{\bar{\partial}}^{1}$, which leads to vanishing the first cohomology group $H_{\Psi}^{1}(N)$. We formulate the injectivity condition in some geometric terms. Here is one of the results in the above context.

\begin{claim}\label{Claim:LieBundle}
Let $S\subset M$ be an embedded symplectic leaf of the Poisson manifold $(M,\Psi)$ such that the normal bundle of $S$ (viewing as a Lie-Poisson bundle) is trivial. Assume that the isotropy algebra of the symplectic leaf $S$ is a semisimple Lie algebra of compact type. If $S$ is compact and simply connected, then there exists a tubular neighborhood $N$ of $S$ in $M$ such that every Poisson vector field of $\Psi$ is Hamiltonian on $N$.
\end{claim}

Secondly, we apply formula \eqref{S2} to another particular case, when the Casimir functions of the transversal Poisson structure $\Psi_{V}$ are projectable (foliated) with respect to the vertical foliation of the tubular neighborhood. In this case, the cochain complex $(\mathcal{C}^{*},\bar{\partial})$ is isomorphic to the de Rham complex of the symplectic leaf $S$ and hence $H_{\bar{\partial}}^{\ast}\cong H_{\operatorname{dR}}^{\ast}(S)$. Therefore, under assumption that $H_{\operatorname{dR}}^{2}(S)=0$, the first cohomology $H_{\Psi}^{1}(N)$ is directly expressed in terms of the Lie algebra $\mathcal{A}$. Using this argument, we illustrate the computation of $H_{\Psi}^{1}(N)$ by some examples.

We remark that most of the results of this article (for example, Claim \ref{Claim:Semilocal}) can be generalized to the Dirac case in order to describe the first Lie algebroid cohomology of Dirac structures around presymplectic leaves \cite{Marcut-13,Va2-04}.

The paper is organized as follows. In Section \ref{Sec:Cov}, we briefly recall some notions and facts about Ehresmann connections.
In Section \ref{Sec:coupling}, we review some properties of coupling Poisson structures and formulate the result on the
Lichnerowicz-Poisson complex in coupling neighborhoods of symplectic leaves. In Section \ref{Sec:InfPoiss}, we present our
main results on the infinitesimal automorphisms of coupling Poisson structures and give a proof of Claim \ref{Claim:Semilocal}.
In the last three sections, \ref{Sec:Vanish}-\ref{Sec:ProyCasim}, the general results are applied to some particular cases.

\section{Covariant Exterior Derivatives}\label{Sec:Cov}

In this section, we recall some notions and facts in the theory of Ehresmann connections on fiber bundles which will be used throughout the text (for more details, see \cite{KMS-93,Va-04,Vo-05}).

Let $E\overset{\pi}{\rightarrow} B$ be a fiber bundle (a surjective submersion) over a manifold $B$. Denote by $\mathbb{V}:=\ker d\pi\subset TE$ the vertical subbundle and by $\mathbb{V}^{0}$ $\subset T^{\ast}E$ its annihilator. The sections of the vector bundles $\wedge^{q}\mathbb{V}$ and $\wedge^{p}\mathbb{V}^{0}$ are called the vertical $q$-vector fields and horizontal $p$-forms on $E$, respectively. In particular,
$\Gamma(\wedge^{0}\mathbb{V}^{0})=C^{\infty}(E)$.

Recall that a vector field $X\in\mathfrak{X}(E)$ is said to be $\pi$-related to $u\in\mathfrak{X}(B)$, if $d\pi\circ X = u\circ \pi$. In this case, we say that $X$ is $\pi$-\emph{projectable}. It is clear that $\pi$-projectable vector fields form an $\mathbb{R}$-Lie subalgebra of $\mathfrak{X}(E)$, which will be denoted by $\Gamma_{\pi\text{-}\mathrm{pr}}(TE)$. Given any vector subbundle $F\subset TE$, we will also use the notation $\Gamma_{\pi\text{-}\mathrm{pr}}(F) := \Gamma_{\pi\text{-}\mathrm{pr}}(TE)\cap\Gamma(F)$ throughout this text. Note that every $\pi$-projectable vector field $X\in\Gamma_{\pi\text{-}\mathrm{pr}}(TE)$ has the property $[X,\Gamma(\mathbb{V})]\subseteq\Gamma(\mathbb{V})$. Conversely, in the case when the fibers of $E\overset{\pi}{\rightarrow} B$ are connected, this property characterizes the $\pi$-projectability of vector fields.

On the other hand, recall that an \emph{Ehresmann connection} on $E$ is a vector bundle morphism $\gamma:TE\rightarrow TE$ such that $\gamma^{2}=\gamma$ and is identical on the vertical subbundle, $\gamma(Y)=Y$ \ $\forall~Y\in\Gamma(\mathbb{V})$. Therefore, $\gamma$ induces a splitting
\begin{equation}\label{Hor1}
TE=\mathbb{H}\oplus\mathbb{V},
\end{equation}
where $\mathbb{H}=\mathbb{H}^{\gamma}:=\ker\gamma$ is the horizontal subbundle associated with $\gamma$. Conversely, given a subbundle $\mathbb{H}$ complementary to $\mathbb{V}$, one can recover the Ehresmann connection from $\mathbb{H}$ by setting $\gamma=\operatorname{pr}_{V}:TE\rightarrow \mathbb{V}$ (the projection along $\mathbb{H}$).

Suppose we are given an Ehresmann connection $\gamma$. The smooth sections of the vector bundles $\wedge^{p}\mathbb{H}$ and $\wedge^{q}\mathbb{H}^{0}$ are said to be horizontal $p$-vector fields and the vertical $q$-forms on $E$, respectively. The \emph{horizontal lift} of $u\in\mathfrak{X}(B)$ with respect to $\gamma$ is the unique horizontal vector field $\operatorname{hor}^{\gamma}(u)\in\Gamma(\mathbb{H})$ which is $\pi$-related with $u$. Therefore, the horizontal lifts are $\pi$-projectable and hence, satisfy the following condition:
\begin{equation}\label{Pr}
[\operatorname{hor}^{\gamma}(u),\Gamma(\mathbb{V})]\subset\Gamma(\mathbb{V}).
\end{equation}

Consider the $C^{\infty}(B)$-module $\Omega_{B}^{p,q}(E):= \Omega^{p}(B)\otimes_{C^{\infty}(B)}\Gamma(\wedge^{q}\mathbb{V})$ of all $p$-forms on $B$ with values in vertical $q$-vector fields on $E$. In particular, by property \eqref{Pr}, the Lie derivative $L_{\operatorname{hor}^{\gamma}(u)}$ leaves invariant the subspaces $\Gamma(\wedge^{q}\mathbb{V})$ of vertical tensor fields. Hence, the $\gamma$-\emph{covariant exterior derivative} $\partial_{1,0}^{\gamma}:\Omega_{B}^{p,q}(E)\rightarrow\Omega_{B}^{p+1,q}(E)$ is defined by the standard formula
\begin{align}
& (\partial_{1,0}^{\gamma}\eta)(u_{0},u_{1},\ldots,u_{p}):=
\sum_{i=0}^{p}(-1)^{i}L_{\operatorname{hor}^{\gamma}(u_{i})}\eta(u_{0},u_{1},\ldots,\hat{u}_{i},\dots,u_{p})\label{Def}\\
&  +\sum_{0\leq i<j\leq p} (-1)^{i+j}
\eta([u_{i},u_{j}],u_{0},\ldots,\hat{u}_{i},\ldots,\hat{u}_{j},\ldots,u_{p}).\nonumber
\end{align}
The curvature form $\operatorname{Curv}^{\gamma}\in\Omega_{B}^{2,1}(E)$ of the connection $\gamma$ is given by
\[
\operatorname{Curv}^{\gamma}(u_{1},u_{2}) :=
[\operatorname{hor}^{\gamma}(u_{1}),\operatorname{hor}^{\gamma}(u_{2})]
- \operatorname{hor}^{\gamma}([u_{1},u_{2}]).
\]
The Bianchi identity reads $\partial_{1,0}^{\gamma}\operatorname{Curv}^{\gamma}=0.$ Moreover, we have the identity
\begin{align}
&((\partial_{1,0}^{\gamma})^{2}\eta)(u_{0},\ldots,u_{p+1})\label{Cur2}\\
&  =-\sum_{0\leq i<j\leq p+1}(-1)^{i+j}L_{\operatorname{Curv}^{\gamma}(u_{i},u_{j})} \eta(u_{0},u_{1},\ldots,\hat{u}_{i},\ldots,\hat{u}_{j},\ldots,u_{p+1})\nonumber
\end{align}
which says that $\partial_{1,0}^{\gamma}$ is a coboundary operator if and only if the connection $\gamma$ is \emph{flat}, i.e.,
$\operatorname{Curv}^{\gamma}=0$. Geometrically, the zero curvature condition is equivalent to the integrability of the horizontal subbundle $\mathbb{H}$.

The splitting \eqref{Hor1} induces the following $\mathbb{H}$-dependent bigrading of multivector fields on $E$:
\begin{equation}\label{Dec}
\Gamma(\wedge^{k}TE)=\bigoplus_{p+q=k}\Gamma(\wedge^{p,q}TE),
\end{equation}
where $\wedge^{p,q}TE:=\wedge^{p}\mathbb{H}\otimes\wedge^{q}\mathbb{V}$. For any $k$-vector field $A$ on $E$, the term of bidegree $(p,q)$ in decomposition \eqref{Dec} is denoted by $A_{p,q}$. Moreover, the dual splitting $T^{\ast}E=\mathbb{V}^{0}\oplus\mathbb{H}^{0}$ induces a bigrading of differential forms on $E$, as follows:
\[
\Gamma(\wedge^{k}T^{\ast}E)=\bigoplus_{p+q=k}\Gamma(\wedge^{p}\mathbb{V}^{0}\otimes\wedge^{q}\mathbb{H}^{0}).
\]

We observe that there exists a natural identification
\[
\Omega_{B}^{p,0}(E)=\Omega^{p}(B)\otimes_{C^{\infty}(B)}C^{\infty}(E)\cong\Gamma(\wedge^{p}\mathbb{V}^{0}).
\]
Indeed, one can associate to every $\eta\in\Omega_{B}^{p,0}(E)$ a horizontal $p$-form
$\pi^{\ast}\eta\in\Gamma(\wedge^{p}\mathbb{V}^{0})$, given for $X_{1},\ldots,X_{p}\in\mathfrak{X}(E)$ and $e\in E$ by
\[
(\pi^{\ast}\eta)(X_{1},\ldots,X_{p})(e) := \eta(d_{e}\pi(X_{1}),\ldots,d_{e}\pi(X_{p})).
\]
Since $\mathbb{V}=\ker d\pi$, it is clear that $\pi^{\ast}\eta\in\Gamma(\wedge^{p}\mathbb{V}^{0})$. Therefore, if we fix an Ehresmann connection $\gamma$, then $\pi^{\ast}\eta$ is uniquely determined by its values on horizontal lifts, namely,
\[
(\pi^{\ast}\eta)(\operatorname{hor}^{\gamma}(u_{1}),\ldots,\operatorname{hor}^{\gamma}(u_{p}))=\eta(u_{1},\ldots,u_{p}).
\]
Moreover, from \eqref{Def}, we get the relation
\begin{equation}\label{PROP}
\pi^{\ast}(\partial_{1,0}^{\gamma}\eta)=(d(\pi^{\ast}\eta))_{p+1,0},
\end{equation}
where $d$ is the exterior differential for forms on $E$.

\section{Coupling Neighborhoods}\label{Sec:coupling}

In this section, we recall some properties of coupling Poisson structures on fiber bundles and their applications to describe the geometry of Poisson manifolds around its symplectic leaves. For more details, see \cite{Marcut-13,Va-04,Vo-01,Vo-05}.

\paragraph{Coupling Poisson Structures} Let $E\overset{\pi}{\rightarrow} B$ be a fiber bundle and $\mathbb{V}^{0}\subset T^{*}E$ the annihilator of the vertical subbundle $\mathbb{V}$.

\begin{definition}
The Poisson structure defined by a bivector field $\Pi\in\mathfrak{X}^{2}(E)$ is said to be a \emph{coupling Poisson structure} on the fiber bundle if
\begin{equation}\label{SPL}
TE=\mathbb{H}\oplus\mathbb{V},\text{~where~}\mathbb{H}:=\Pi^{\sharp}(\mathbb{V}^{0}).
\end{equation}
\end{definition}

Note that every coupling Poisson structure $\Pi$ has the bigraded decomposition of the form $\Pi=\Pi_{2,0}+\Pi_{0,2}$, \ where $\Pi_{2,0}\in\Gamma(\wedge^{2}\mathbb{H})$ is a horizontal bivector field of constant rank, $\operatorname{rank}\Pi_{2,0}=\operatorname{rank}\mathbb{H}$, and $\Pi_{0,2} \in\Gamma(\wedge^{2}\mathbb{V})$ is a \emph{vertical Poisson tensor}. The characteristic distribution of $\Pi$ is the direct sum of the horizontal bundle $\mathbb{H}$ and the
characteristic distribution of $\Pi_{0,2}$,
\[
\Pi^{\sharp}(T^{\ast}E)=\mathbb{H}\oplus\Pi_{0,2}^{\sharp}(\mathbb{H}^{0}).
\]
It follows that the fibers of the projection $\pi$ intersect the symplectic leaves of $\Pi$ transversally and symplectically. Moreover, the restriction of $\Pi_{2,0}^{\sharp}:T^{\ast}E\rightarrow TE$ to $\mathbb{V}^{0}$ is a vector bundle isomorphism onto $\mathbb{H}$.

One can associate to a given coupling Poisson tensor $\Pi$ the geometric data $(\gamma,\sigma,P)$ consisting of the Ehresmann connection $\gamma\in\Omega^{1}(E;\mathbb{V})$ associated with the horizontal subbundle $\mathbb{H} =\Pi^{\sharp}(\mathbb{V}^{0})$, a nondegenerated 2-form $\sigma\in\Omega ^{2}(B)\otimes_{C^{\infty}(B)} C^{\infty}(E)$, called the \emph{coupling form}, and the vertical Poisson bivector field  $P:=\Pi_{0,2}\in\Omega_{B}^{0,2}(E)$. The nondegeneracy of the 2-form $\sigma$ means that the vector bundle morphism  ($\pi^{\ast}\sigma)^{\flat}:\mathbb{H}\rightarrow\mathbb{V}^{0}$ \ is an isomorphism. In terms of the horizontal part of $\Pi$, the coupling form is given by $(\pi^{\ast}\sigma)^{\flat}=-\left(\Pi_{2,0}^{\sharp}\mid_{\mathbb{V}^{0}}\right)^{-1}$. One can show that the geometric data satisfy the \emph{structure equations}
\begin{align}
[P,P]&=0, \label{CC1}\\
L_{\operatorname{hor}^{\gamma}(u)}P&=0, \label{CC2}\\
\operatorname{Curv}^{\gamma}(u,v)&=-P^{\sharp}d\sigma(u,v), \label{CC3}\\
\partial_{1,0}^{\gamma}\sigma&=0, \label{CC4}
\end{align}
for any $u,v\in\mathfrak{X}(B)$, which give a factorization of the Jacobi identity for $\Pi$. Condition \eqref{CC2} means that the connection $\gamma$ on the Poisson fiber bundle $(E\xrightarrow{\pi}B,P)$ is Poisson. In general, the curvature $\operatorname{Curv}^{\gamma}\in\Omega_{B}^{2,1}(E)$ of a Poisson connection takes values in the space of vertical Poisson vector fields of $P$. The curvature identity \eqref{CC3} says that $\operatorname{Curv}^{\gamma}(u,v)$ is a Hamiltonian vector field for any $u,v\in\mathfrak{X}(B)$. Moreover, the coupling 2-form $\sigma$ must be $\gamma$-covariantly constant (condition \eqref{CC4}). We say that some geometric data are \emph{integrable} if they satisfy \eqref{CC1}-\eqref{CC4}.

Conversely, every integrable geometric data $(\gamma,\sigma,P)$ defines a coupling Poisson tensor $\Pi$ on $E$ under the nondegeneracy condition for $\sigma$.

\paragraph{Bigrading of the Lichnerowicz-Poisson Complex} Following \cite{CrFe-10}, let us associate to the geometric data $(\gamma,\sigma,P)$ of a coupling Poison tensor $\Pi\in\mathfrak{X}^{2}(E)$ the following cochain complex. Consider the Schouten-Nijenhuis bracket $[~,~]:\Gamma(\wedge^{k_{1}}TE)\times\Gamma(\wedge^{k_{2}}TE)\rightarrow\Gamma(\wedge^{k_{1}+k_{2}-1}TE)$ for multivector fields on the total space $E$ defined in such a way that the triple $(\Gamma(\wedge TE),\wedge,[~,~])$ is a graded Poisson algebra of degree $-1$ (see \cite{DZ}). It is clear that the Schouten-Nijenhuis bracket of two vertical multivector fields on $E$ is again vertical. As a consequence, we can endow the bigraded $C^{\infty}(B)$-module
\[
\mathfrak{M}^{\ast\ast}=\bigoplus_{k=0}^{\infty}\mathfrak{M}^{k},\qquad \mathfrak{M}^{k}:=\bigoplus_{p+q=k}\Omega_{B}^{p,q}(E).
\]
with a structure of graded Poisson algebra of degree $-1$, $(\mathfrak{M}^{\ast\ast},\wedge,[~,~])$. Explicitly, for $\eta\in\Omega_{B}^{p,q}(E)$ and $\theta\in\Omega_{B}^{p',q'}(E)$, we have \cite{CrFe-10,Marcut-13}
\begin{align*}
  (\eta\wedge\theta)(u_{1},\ldots,u_{p+p'})&:=(-1)^{p'q}\sum_{\tau}\operatorname{sgn}(\tau)\eta(u_{\tau(1)},\ldots,u_{\tau(p)})\wedge\theta(u_{\tau(p+1)},\ldots,u_{\tau(p+p')}),\\
  [\eta,\theta](u_{1},\ldots,u_{p+p'})&:=(-1)^{p'(q-1)}\sum_{\tau}\operatorname{sgn}(\tau)[\eta(u_{\tau(1)},\ldots,u_{\tau(p)}),\theta(u_{\tau(p+1)},\ldots,u_{\tau(p+p')})],
\end{align*}
where $u_{i}\in\mathfrak{X}(B)$. Here, in the right-hand sides of these equalities, the symbols $\wedge$ and $[~,~]$ denote the exterior product and the Schouten-Nijenhuis bracket on $\Gamma(\wedge^{*}\mathbb{V})$, respectively. Thus, every element $\theta\in\Omega_{B}^{p,q}(E)$ induces a graded derivation $\operatorname{ad}_{\theta}$ of bidegree $(p,q-1)$, defined by the adjoint operator
$\operatorname{ad}_{\theta}(\cdot)=[\theta,\cdot]$. In particular, the vertical Poisson bivector field $P\in\Gamma(\wedge^{2}\mathbb{V})$ induces the derivation $\delta_{P}:=\operatorname{ad}_{P}:\Omega_{B}^{p,q}(E)\rightarrow\Omega_{B}^{p,q+1}(E)$ of bidegree $(0,1)$ given by
\[
(\operatorname{ad}_{P}\eta)(u_{1},\dots,u_{p}):=(-1)^{p}[P,\eta(u_{1},\dots,u_{p})].
\]
This is a coboundary operator which gives rise to the vertical Poisson complex $(\oplus_{q=0}^{\infty}\Omega_{B}^{0,q}(E),\delta_{P})$.

Now, using the geometric data $(\gamma,\sigma,P)$, we can define an operator $\partial:\mathfrak{M}^{\ast\ast}\rightarrow\mathfrak{M}^{\ast\ast}$ as the sum of bigraded operators
\begin{equation}\label{Co}
\partial := \partial_{2,-1}^{\sigma} + \partial_{1,0}^{\gamma} + \partial_{0,1}^{P},
\end{equation}
where $\partial_{2,-1}^{\sigma}:=-\operatorname{ad}_{\sigma}$, \ $\partial_{1,0}^{\gamma}$ is the covariant exterior derivative (see Section \ref{Sec:Cov}), and $\partial_{0,1}^{P}:=\delta_{P}$. Observe that the integrability conditions for the geometric data $(\gamma,\sigma,P)$ mean that $\partial$ is a coboundary operator, $\partial^{2}=0$. Indeed, computing the bigraded components of $\partial^{2}$, we get that equations \eqref{CC1}-\eqref{CC4} are equivalent to the following relations:
\begin{align}
  (\partial_{0,1}^{P})^{2} &=0,\label{eq:Cob1}\\
  \partial_{1,0}^{\gamma}\partial_{0,1}^{P} + \partial_{0,1}^{P}\partial_{1,0}^{\gamma} &=0,\label{eq:Cob2}\\
  \partial_{2,-1}^{\sigma}\partial_{0,1}^{P} + \partial_{0,1}^{P}\partial_{2,-1}^{\sigma} + (\partial_{1,0}^{\gamma})^{2} &=0,\label{eq:Cob3}\\
  \partial_{2,-1}^{\sigma}\partial_{1,0}^{\gamma} + \partial_{1,0}^{\gamma}\partial_{2,-1}^{\sigma} &= 0.\label{eq:Cob4}
\end{align}
Moreover, by the Jacobi identity for the bracket on $\mathfrak{M}$, one can show that $(\partial_{2,-1}^{\sigma})^{2} = 0$.

Various versions of the following fact can be found in \cite{CrFe-10,Marcut-13}.

\begin{proposition}\label{prop:IsoCompx}
Let $\Pi\in\mathfrak{X}^{2}(E)$ be a coupling Poisson tensor on $E\overset{\pi}{\rightarrow}B$ and let $(\gamma,\sigma,P)$ be the geometric data associated with $\Pi$. Then the Lichnerowicz-Poisson complex $(\mathfrak{X}^{\ast}(E),\delta_{\Pi})$ is isomorphic to the cochain complex $(\mathfrak{M}^{\ast\ast},\partial)$.
\end{proposition}

\begin{proof}
Consider the decomposition of multivector fields \eqref{Dec}. Note that each $A\in\Gamma(\wedge^{p,q}TE)$ can be viewed as a $C^{\infty}(B)$ $p$-linear skew-symmetric map $A:\Gamma(\mathbb{V}^{0})\times\cdots\times\Gamma(\mathbb{V}^{0})\rightarrow\Gamma(\wedge^{q}\mathbb{V})$.
Define $\flat_{\sigma}A\in\mathfrak{M}^{p,q}$ by
\begin{equation}\label{eq:IsoCompx}
  (\flat_{\sigma}A)(u_{1},\ldots,u_{p}):=(-1)^{p}A((\pi^{*}\sigma)^{\flat}\operatorname{hor}^{\gamma}u_{1},\ldots,(\pi^{*}\sigma)^{\flat}\operatorname{hor}^{\gamma}u_{p}).
\end{equation}
for any $u_{i}\in\mathfrak{X}(B)$. We claim that the map $\flat_{\sigma}:\Gamma(\wedge TE)\rightarrow\mathfrak{M}$ is a cochain complex isomorphism. Since $(\pi^{*}\sigma)^{\flat}|_{\mathbb{H}}=-\left(\Pi_{2,0}|_{\mathbb{V}^{0}}\right)^{-1}$ is a vector bundle isomorphism, it follows that $\flat_{\sigma}$ is an exterior algebra isomorphism. By the property that every graded derivation of $\Gamma(\wedge TE)$ is determined by its action on $C^{\infty}(E)$ and $\Gamma(TE)$, it suffices to show that $\flat_{\sigma}\circ\delta_{\Pi} = \partial\circ\flat_{\sigma}$ holds on $C^{\infty}(E)$, $\Gamma(\mathbb{V})$ and $\Gamma(\mathbb{H})$.

For every $f\in C^{\infty}(E)$, we have $(\flat_{\sigma}\circ\delta_{\Pi})(f) = \flat_{\sigma}[\Pi_{2,0},f]+[P,f]$ and $(\partial\circ\flat_{\sigma})(f) = \partial^{\gamma}_{1,0}f + [P,f]$. Moreover, $\flat_{\sigma}[\Pi_{2,0},f](u) = df(\operatorname{hor}^{\gamma}u)=\partial^{\gamma}_{1,0}f(u)$.

Next, let $X\in\Gamma(\mathbb{V})$. By bigrading arguments, the equality $(\flat_{\sigma}\circ\delta_{\Pi})(X) = (\partial\circ\flat_{\sigma})(X)$ splits into three equations: $\flat_{\sigma}[\Pi,X]_{1,1} = \partial^{\gamma}_{1,0}X$, $\flat_{\sigma}[\Pi,X]_{0,2} = \delta_{P}X$, and $\flat_{\sigma}[\Pi,X]_{2,0} = -\mathrm{ad}_{\sigma}X$. For the first equation, by definition, we have $\flat_{\sigma}[\Pi,X]_{1,1}(u) = [\operatorname{hor}^{\gamma}u,X] = \partial^{\gamma}_{1,0}X(u)$. The second one holds because of $[\Pi_{2,0},X]_{0,2}=0$. The last equation follows from $\flat_{\sigma}[\Pi,X]_{2,0}(u,v) = [X,\sigma(u,v)] = -\mathrm{ad}_{\sigma}X(u,v)$.

Finally, for $X=\operatorname{hor}^{\gamma}u$, $u\in\mathfrak{X}(B)$, the equality $(\flat_{\sigma}\circ\delta_{\Pi})(X) = (\partial\circ\flat_{\sigma})(X)$ splits into the following relations: $\partial^{\gamma}_{1,0}\flat_{\sigma}(X) = \flat_{\sigma}[\Pi,X]_{2,0}$, $\delta_{P}\flat_{\sigma}(X) = \flat_{\sigma}[\Pi,X]_{1,1}$, and $0=\flat_{\sigma}[\Pi,X]_{0,2}$. The verification of these equalities is straightforward by using the structure equations \eqref{CC1}-\eqref{CC4}.
\end{proof}

As a consequence of Proposition \ref{prop:IsoCompx}, we conclude that the infinitesimal automorphisms of the coupling Poisson structure $\Pi$ are determined by the 1-cocycles $\eta=\eta_{1,0}+\eta_{0,1}\in Z_{\partial}^{1}$ of $\partial$ which are the solutions to the equations:
\begin{align}
\partial_{0,1}^{P}(\eta_{0,1})&=0, \label{EQ1}\\
\partial_{1,0}^{\gamma}(\eta_{0,1})+\partial_{0,1}^{P}(\eta_{1,0})&=0, \label{EQ2}\\
\partial_{1,0}^{\gamma}(\eta_{1,0})+\partial_{2,-1}^{\sigma}(\eta_{0,1})&=0.\label{EQ3}
\end{align}
In the next section, we describe the infinitesimal automorphisms of coupling Poisson structures in terms of the solutions of these equations.

\paragraph{Coupling Neighborhood of a Symplectic Leaf} Let $(M,\Psi)$ be a Poisson manifold and $B\subset M$ an embedded symplectic leaf. Let $\pi:E\rightarrow B,$ $E=T_{B}M\diagup TB$ be the normal bundle of the leaf. By a tubular neighborhood of a symplectic leaf $B$, we mean an open neighborhood $N$ of $B$ in $M$ together with an exponential map $\mathbf{f}:U\rightarrow N$, that is, a diffeomorphism from an open neighborhood $U$  of the zero section $B\hookrightarrow E$ onto $N$ satisfying the conditions: $\mathbf{f}|_{B}=\operatorname{id}_{B}$ and $\nu\circ d_{B}\mathbf{f}=\tau $. Here, $\tau:T_{B}E\rightarrow E$ is the projection along $TB$ according to the decomposition $T_{B}E=TB\oplus E$ and  $\mathbf{\nu}:T_{B}M\rightarrow E$ is the natural projection. These properties imply that the differential $d_{B}\mathbf{f}:T_{B}E\rightarrow  T_{B}M$ sends the fibers of the normal bundle to transverse subspaces to the leaf $B\subset M$, $T_{B}M=TB\oplus(d_{B}\mathbf{f})(E)$.

\begin{definition}
A tubular neighborhood $(N,\mathbf{f})$ of the symplectic leaf $B$ of $(M,\Psi)$ is said to be a \emph{coupling neighborhood} if the pull-back $\Pi:=\mathbf{f}^{\ast}(\Psi|_{N})$ is a coupling Poisson structure on the fiber bundle $\pi_{U}:U\rightarrow B$.
\end{definition}

Given a coupling neighborhood $(N,\mathbf{f})$ of $B$, we have the bigraded decomposition $\Pi=\Pi_{2,0}+\Pi_{0,2}$. Hence, $\Psi|_{N}=\Psi_{H}+\Psi_{V}$, where $\Psi_{H}=\mathbf{f}_{\ast}\Pi_{2,0}$ is a bivector field on $N$ of constant rank, $\operatorname{rank}\Psi_{H}=\dim B$, and $\Psi_{V}=$ $\mathbf{f}_{\ast}\Pi_{0,2}$ is a Poisson tensor on $N$ vanishing at $B$ and tangent to the vertical subbundle $d\mathbf{f}(\ker d\pi_{U})\subset T_{N}M$ over the tubular neighborhood. The bivector field $\Psi_{V}$ is said to be a \emph{transverse Poisson structure} around the leaf $B$ and can be viewed as the result of gluing the local transverse Poisson structures on the vertical fibers due to the local splitting Weinstein theorem \cite{We-83}. Furthermore, one can show that the different choices of exponential maps lead to isomorphic transverse Poisson structures. Notice that, in the case when the symplectic leaf $B$ is regular, the coupling neighborhood $N$ may be chosen in such a way that the transverse Poisson structure is identically zero, $\Psi_{V}\equiv0$. This follows from the property: $\operatorname{rank}_{m}\Psi=\dim B + \operatorname{rank}_{m}\Psi_{V}$ for every $m\in N$. Observe also that the linearization of $\Pi_{0,2}$ at $B$ gives a vertical fiberwise linear Poisson structure $\Pi_{0,2}^{(1)}$. This Lie algebra is called the \emph{linearized transverse Poisson structure} of the leaf $B$ \cite{We-83}, which is well defined on the whole total space $E$. As a consequence, we get an intrinsic locally trivial Lie-Poisson
bundle $(E,\Pi_{0,2}^{(1)})$ over $B$ whose typical fiber is the co-algebra $\mathfrak{g}^{\ast}$ of a Lie algebra $\mathfrak{g}$ called the \emph{isotropy algebra} of the symplectic leaf.

As is known \cite{Vo-01}, each embedded symplectic leaf $B$ admits a coupling neighborhood and hence, by Proposition \ref{prop:IsoCompx}, the computation of the Poisson cohomology around $B$ is reduced to the study of the bigraded cochain complex $(\mathfrak{M}^{\ast\ast},\partial)$ attributed to a coupling Poisson structure $\Pi$.

\section{Infinitesimal Automorphisms of Coupling Poisson Structures}\label{Sec:InfPoiss}

Suppose we are given a coupling Poisson tensor $\Pi$ on a fiber bundle $\pi:E\rightarrow B$ associated with an integrable geometric data $(\gamma,\sigma,P)$. As we saw in the previous section, the infinitesimal Poisson automorphisms of $\Pi$ are related to the solutions of equations \eqref{EQ1}-\eqref{EQ3}. Our goal is to describe these solutions in terms of the geometric data $(\gamma,\sigma,P)$. To formulate the main results, let us introduce the following objects.

\paragraph{The coboundary operator $\bar{\partial}^{\gamma}$} Consider the space $\operatorname{Casim}(E,P)$ of all Casimir functions of the vertical Poisson tensor $P$ on $E$. It is clear that $\pi^{\ast}C^{\infty}(B)\subseteq\operatorname{Casim}(E,P)$. Define the $C^{\infty}(B)$-submodules $\mathcal{C}^{p}\subseteq\Omega_{B}^{p,0}(E)$ of the form
\[
\mathcal{C}^{p}:=\Omega^{p}(B)\otimes_{C^{\infty}(B)}\operatorname{Casim}(E,P).
\]
In particular, $\mathcal{C}^{0}=\operatorname{Casim}(E,P)$. Since the Poisson vector fields of $P$ preserve the space of Casimir functions, by  \eqref{CC2} and definition \eqref{Def}, we have $\partial_{1,0}^{\gamma}(\mathcal{C}^{p})\subset\mathcal{C}^{p+1}$. Hence one can define the  operator
\begin{equation}\label{DEF}
\bar{\partial}^{\gamma}:=\partial_{1,0}^{\gamma}|_{\mathcal{C}^{p}}.
\end{equation}
Then, by \eqref{Cur2} and the curvature identity \eqref{CC3}, we conclude that $\bar{\partial}^{\gamma}$ is a coboundary operator, $\bar{\partial}^{\gamma}\circ\bar{\partial}^{\gamma}=0$. The $p$-cohomology space of $\bar{\partial}^{\gamma}$ is $H_{\bar{\partial}^{\gamma}}^{p} := \frac{Z_{\bar{\partial}^{\gamma}}^{p}}{B_{\bar{\partial}^{\gamma}}^{p}}$, where $Z_{\bar{\partial}^{\gamma}}^{p}$ and $B_{\bar{\partial}^{\gamma}}^{p}$ are the spaces of $\bar{\partial}^{\gamma}$-closed and $\bar{\partial}^{\gamma}$-exact $p$-forms, respectively.

Consider the Lie algebra $\operatorname{Poiss}(E,\Pi)$ of Poisson vector fields of the coupling Poisson structure $\Pi$. Let $\sharp_{H}:\Omega_{B}^{1,0}(E)\rightarrow\Gamma(\mathbb{H})$ be a linear mapping given by
\[
\sharp_{H}(\alpha):=\Pi_{2,0}^{\sharp}(\pi^{\ast}\alpha).
\]
By the horizontal nondegeneracy of $\Pi_{2,0}$, it follows that $\sharp_{H}:\Omega_{B}^{1,0}(E)\rightarrow\Gamma(\mathbb{H})$ is an isomorphism.

\begin{lemma}\label{lemma:IsoPoissHor}
The image of the space of 1-cocycles $Z_{\bar{\partial}^{\gamma}}^{1}$ under the isomorphism $\sharp_{H}$ coincides with the space of Poisson vector fields of $\Pi$ tangent to the horizontal distribution,
\begin{equation}\label{IJ}
\sharp_{H}(Z_{\bar{\partial}^{\gamma}}^{1})=\Gamma(\mathbb{H})\cap\operatorname{Poiss}(E,\Pi).
\end{equation}
\end{lemma}

\begin{proof}
Consider the vector bundle morphism $\Pi^{\sharp}:T^{\ast}M\rightarrow TM$ associated with the bivector field $\Pi$. For every $2$-form $\mu$ $\in\Omega^{2}(E)$, one can associate a bivector field $\Pi^{\sharp}\mu\in\mathfrak{X}^{2}(E)$ defined by $(\Pi^{\sharp}\mu)(\eta_{1},\eta_{2}) :=
\mu(\Pi^{\sharp}\eta_{1},\Pi^{\sharp}\eta_{2})$. Then, one has $[\Pi,\Pi^{\sharp}(\pi^{\ast}\alpha)]=-\Pi^{\sharp}(d\pi^{\ast}\alpha)$.
Moreover, we observe that, for any $\eta\in\Omega^{1}(E)$ and $\mu\in\Omega^{2}(E)$ such that $\eta_{0,1}=0$ and $\mu_{0,2}=0$, the following identities hold: $\Pi^{\sharp}\eta=\Pi_{2,0}^{\sharp}\eta$ and $\Pi_{2,0}^{\sharp}\mu=\Pi^{\sharp}_{2,0}\mu_{2,0}$. Setting $\eta=\pi^{\ast}\alpha$, and $\mu=d\pi^{\ast}\alpha$ and combining these properties with \eqref{PROP}, we get
\begin{equation}\label{eq:LemaSharp1}
[\Pi,\sharp_{H}(\alpha)] = [\Pi,\Pi_{2,0}^{\sharp}(\pi^{\ast}\alpha)] = [\Pi,\Pi^{\sharp}(\pi^{\ast}\alpha)]= -\Pi^{\sharp}(d\pi^{\ast}\alpha)
\end{equation}
and
\begin{equation}\label{eq:LemaSharp2}
\Pi^{\sharp}_{2,0}(\pi^{\ast}(\bar{\partial}^{\gamma}\alpha)) = \Pi^{\sharp}_{2,0}(d\pi^{\ast}\alpha)_{2,0} = \Pi^{\sharp}_{2,0}(d\pi^{\ast}\alpha).
\end{equation}
Finally, observe that $\alpha\in\mathcal{C}^{1}$ if and only if $\mathrm{i}_{P^{\sharp}\eta}d\pi^{*}\alpha=0$ $\forall\eta\in\Omega^{1}(E)$, which is equivalent to $\Pi^{\sharp}(d\pi^{\ast}\alpha)=\Pi^{\sharp}_{2,0}(d\pi^{\ast}\alpha)$. Therefore, from \eqref{eq:LemaSharp1} and \eqref{eq:LemaSharp2}, it follows that $[\Pi,\sharp_{H}(\alpha)] =0$ if and only if $\bar{\partial}^{\gamma}\alpha=0$.
\end{proof}

\begin{remark}
Notice that Lemma \ref{lemma:IsoPoissHor} can be deduced from Proposition \ref{prop:IsoCompx}. Indeed, the cochain complex isomorphism
$\flat_{\sigma}:\Gamma(\wedge TE)\rightarrow\mathfrak{M}$ satisfies $\sharp_{H} = -(\flat_{\sigma}|_{\Gamma(\mathbb{H})})^{-1}$. Thus, $\sharp_{H}(\Omega^{1,0}_{B}(E)\cap Z^{1}_{\partial}) = \Gamma(\mathbb{H})\cap\operatorname{Poiss}(E,\Pi)$. Since $Z_{\bar{\partial}^{\gamma}}^{1}=\Omega^{1,0}_{B}(E)\cap Z^{1}_{\partial}$, the result follows.
\end{remark}

\paragraph{The Lie Algebra $\mathcal{A}^{\gamma}$} Let $\operatorname{Ham}(E,P)\subset\Gamma(\mathbb{V})$ be the Lie algebra of Hamiltonian vector fields of the vertical Poisson tensor $P$ on $E$. Consider the Poisson connection $\gamma$ on $(E,P)$. The set of all vertical Poisson vector fields is a Lie algebra
\[
\operatorname{Poiss}_{V}(E,P):=\{Y\in\Gamma(\mathbb{V})\mid L_{Y}P=0\}
\]
for which $\operatorname{Ham}(E,P)$ is an ideal. Furthermore, by \eqref{Pr} and \eqref{CC2}, we have
\[
[\operatorname{hor}^{\gamma}(u),\operatorname{Poiss}_{V}(E,P)]\subseteq\operatorname{Poiss}_{V}(E,P)\qquad\forall u\in\mathfrak{X}(B).
\]
One can associate to the triple $(E,P,\gamma)$ the subspace $\mathcal{A}^{\gamma}\mathcal{\subset}\operatorname{Poiss}_{V}(E,P)$ of vertical Poisson vector fields determined by the condition
\[
[\operatorname{hor}^{\gamma}(u),\mathcal{A}^{\gamma}]\subseteq\operatorname{Ham}(E,P)\qquad\forall u\in\mathfrak{X}(B)
\]
or, more precisely,
\begin{equation}\label{ALP}
\mathcal{A}^{\gamma}:=\{Y\in\operatorname{Poiss}_{V}(E,P)\mid[\operatorname{hor}^{\gamma}(u),Y]\in\operatorname{Ham}(E,P)\quad\forall u\in\mathfrak{X}(B)\}.
\end{equation}
Observe that $\mathcal{A}^{\gamma}$ is a Lie algebra and $\operatorname{Ham}(E,P)\subseteq\mathcal{A}^{\gamma}$ is an ideal. These properties follow from the identity
\[
[X,P^{\sharp}dF]=P^{\sharp}dL_{X}F
\]
for any Poisson vector field $X$ of $P$. Moreover, for every $Y\in\mathcal{A}^{\gamma}$ there exists a 1-form $\beta_{Y}\in\Omega_{B}^{1,0}(E)=\Omega^{1}(B)\otimes_{C^{\infty}(B)} C^{\infty}(E)$ such that
\begin{equation}\label{3F}
[\operatorname{hor}^{\gamma}(u),Y]=-P^{\sharp}d\beta_{Y}(u) \qquad\forall u\in\mathfrak{X}(B).
\end{equation}
This follows from a partition of unity argument applied to an open coordinate covering of the base $B$, and the fact that $P$ is vertical.

\paragraph{The homomorphism $\rho^{\gamma}:\mathcal{A}^{\gamma}\rightarrow H_{\bar{\partial}^{\gamma}}^{2}$} Given an arbitrary vector field $Y\in\mathcal{A}^{\gamma}$ and fixing a 1-form $\beta_{Y}\in\Omega^{1}(B)\otimes_{C^{\infty}(B)} C^{\infty}(E)$ in \eqref{3F}, we associate to $Y$ an element $\tau_{Y}\in\Omega^{2}(B)\otimes_{C^{\infty}(B)} C^{\infty}(E)$ given by
\begin{equation}\label{2F}
\tau_{Y}:=\partial_{1,0}^{\gamma}\beta_{Y}+L_{Y}\sigma.
\end{equation}
Here, the Lie derivative $L_{Y}:\Omega_{B}^{p,q}(E) \rightarrow \Omega_{B}^{p,q}(E)$ along an arbitrary vertical vector field $Y$
is given by the standard formula $(L_{Y}\eta)(u_{1},\ldots,u_{k}):= L_{Y}\eta(u_{1},\ldots,u_{k})$. Note that we also have the equality $L_{Y}\sigma = \partial_{2,-1}^{\sigma}Y$.

By using the structure equations \eqref{CC1}-\eqref{CC4}, one can show that the 2-form $\tau_{Y}$ takes values in Casimir functions, $\tau_{Y}\in\mathcal{C}^{2}$. Indeed, from \eqref{2F}, we have
\begin{equation}\label{4F}
\tau_{Y}(u_{1},u_{2})  = L_{\operatorname{hor}^{\gamma}(u_{1})}\beta_{Y}(u_{2}) -
L_{\operatorname{hor}^{\gamma}(u_{2})}\beta_{Y}(u_{1}) - \beta_{Y}([u_{1},u_{2}]) + L_{Y}\sigma(u_{1},u_{2}).
\end{equation}
Next, for every Poisson vector field $Z$ of $P$, we have $L_{Z}\circ P^{\sharp}=P^{\sharp}\circ L_{Z}$. \ Using this property, equality \eqref{4F}, the curvature identity \eqref{CC3} and \eqref{3F}, by direct computation we verify that $P^{\sharp}d\tau_{Y}(u_{1},u_{2})=0$.

Now, we observe that the 2-form $\tau_{Y}$ is $\bar{\partial}^{\gamma}$-closed, $\bar{\partial}^{\gamma}\tau_{Y}=0$. Indeed, this can be verified by straightforward computations and by applying again \eqref{2F}, \eqref{Cur2}, \eqref{CC3}, \eqref{3F}, and \eqref{CC4}. Moreover, the cohomology class $[\tau_{Y}]\in H_{\bar{\partial}^{\gamma}}^{2}$ is independent of the choice of $\beta_{Y}$ in \eqref{3F}. To see this, observe that any other element $\beta'_{Y}\in\Omega_{B}^{1,0}(E)$ satisfying \eqref{3F} is of the form $\beta_{Y}'=\beta_{Y}+c_{Y}$ for some $c_{Y}\in\mathcal{C}^{1}$. Then, the corresponding $\tau_{Y}$ and $\tau_{Y}'$ are related by $\tau_{Y}' = \tau_{Y}+\bar{\partial}^{\gamma }c_{Y}$ and hence, $[\tau_{Y}']=[\tau_{Y}]$. So, we have proved the following fact.

\begin{lemma}\label{lemma:DefRho}
There exists an intrinsic homomorphism
\begin{equation}\label{Hom}
\rho^{\gamma}:\mathcal{A}^{\gamma}\rightarrow H_{\bar{\partial}^{\gamma}}^{2}
\end{equation}
which assigns to every vertical vector field $Y\in\mathcal{A}^{\gamma}$ the $\bar{\partial}^{\gamma}$-cohomology class of the 2-form $\tau_{Y}$,
\[
\rho^{\gamma}(Y):=[\tau_{Y}].
\]
\end{lemma}

It is easy to see that every Hamiltonian vector field of $P$ belongs to the kernel of $\rho^{\gamma}$ and hence we have the inclusions: \begin{equation}\label{IN}
\operatorname{Ham}(E,P)\subseteq\ker\rho^{\gamma} \subseteq \mathcal{A}^{\gamma}\subseteq\operatorname{Poiss}_{V}(E,P).
\end{equation}

Consider the projection $\operatorname{pr}_{V}:\mathfrak{X}(E)\rightarrow\Gamma(\mathbb{V})$ associated with the splitting $TE=\mathbb{H}\oplus\mathbb{V}$, $\operatorname{pr}_{V}(X)=X_{0,1}$. It is clear that $\ker \operatorname{pr}_{V}=\Gamma(\mathbb{H})$.
\begin{lemma}\label{lemma:KerRho}
The image of $\operatorname{Poiss}(E,\Pi)$ under the projection $\operatorname{pr}_{V}$ coincides with the kernel of $\rho^{\gamma}$,
\[
\operatorname{pr}_{V}(\operatorname{Poiss}(E,\Pi))=\ker\rho^{\gamma}.
\]
\end{lemma}

\begin{proof}
Let $Z\in\operatorname{Poiss}(E,\Pi)$ be an infinitesimal automorphism of $\Pi$. Since the map $\sharp_{H}:\Omega_{B}^{1,0}(E)\rightarrow\Gamma(\mathbb{H})$ is an isomorphism, there exist unique $Y\in\Gamma(\mathbb{V})$ and $\beta\in\Omega_{B}^{1,0}(E)$ such that $Z=-\sharp_{H}\beta + Y$. Let us show that $Y\in\ker\rho^{\gamma}$. If $\flat_{\sigma}:(\mathfrak{X}^{\ast}(E),\delta_{\Pi})\rightarrow(\mathfrak{M}^{\ast\ast},\partial)$ is the cochain complex isomorphism \eqref{eq:IsoCompx}, then $\flat_{\sigma}Z = \beta + Y$. Since $Z\in Z^{1}_{\Pi}(E)$, we have $\beta+Y\in Z^{1}_{\partial}$. Explicitly, this means that $\eta:=\beta+Y$ must satisfy equations \eqref{EQ1}-\eqref{EQ3}. Note that \eqref{EQ1} means that $Y\in\operatorname{Poiss}_{V}(E,P)$. Moreover, by evaluating the left-hand side of \eqref{EQ2} on $u\in\mathfrak{X}(B)$, we get that $\beta$ and $Y$ satisfy \eqref{3F}, so $Y\in\mathcal{A}^{\gamma}$. Finally, \eqref{EQ3} implies that $\tau_{Y}=0$ and hence, $\rho^{\gamma}(Y)=0$, as desired. Conversely, pick an arbitrary $Y\in\ker\rho^{\gamma}$. Since $\ker\rho^{\gamma}\subseteq\mathcal{A}^{\gamma}$, there exists a 1-form $\beta_{Y}\in\Omega_{B}^{1,0}(E)$ satisfying \eqref{3F}. Next, by the definition of $\rho^{\gamma}$, there exists a primitive $c_{Y}\in\mathcal{C}^{1}$ of the $2$-cocycle $\tau_{Y}\in Z^{2}_{\bar{\partial}^{\gamma}}$ in \eqref{2F} so that $\bar{\partial}^{\gamma}c_{Y}=\tau_{Y}$. Then, one can easily verify that
\begin{equation}\label{VC}
X_{Y}:=-\sharp_{H}(\beta_{Y}-c_{Y})+Y\in\operatorname{Poiss}(E,\Pi).
\end{equation}
This means that every element $Y\in\ker\rho^{\gamma}$ can be extended to a Poisson vector field $X$ of $\Pi$ in the sense that $X_{0,1}=Y$.
\end{proof}

\begin{corollary}
The Poisson vector fields of the coupling Poisson structure $\Pi$ are of the form
\begin{equation}\label{PVF1}
X=\sharp_{H}(\alpha)+X_{Y},
\end{equation}
where $\alpha\in Z_{\bar{\partial}^{\gamma}}^{1}$ and $Y\in\ker \rho^{\gamma}\subset\mathcal{A}^{\gamma}$ \ are
arbitrary elements.
\end{corollary}

In particular, it follows that a Poisson vector field $X\in\operatorname{Poiss}(E,\Pi)$ is tangent to the symplectic foliation of $\Pi$ if and only if $\operatorname{pr}_{V}(X)$ is tangent to the symplectic foliation of the vertical Poisson structure $P$.

By Lemma \ref{lemma:DefRho} and Lemma \ref{lemma:KerRho}, we conclude that $\operatorname{Poiss}(E,\Pi)$ fits into the short
exact sequence of vector spaces
\begin{equation}\label{X1}
0\rightarrow Z_{\bar{\partial}^{\gamma}}^{1}\overset{\sharp_{H}}{\longrightarrow}\operatorname{Poiss}(E,\Pi)\overset{\operatorname{pr}_{V}}{\longrightarrow}\ker\rho^{\gamma}\rightarrow0.
\end{equation}
Summarizing the above considerations, we get the following splitting theorem for infinitesimal automorphisms of a coupling Poisson structure.

\begin{theorem}
Let $\Pi=\Pi_{2,0}+P$ be the coupling Poisson tensor on $E$ and $(\gamma,\sigma,P)$ its integrable geometric data. Let $(\mathcal{A}^{\gamma},\bar{\partial}^{\gamma},\rho^{\gamma})$ be the associated set up. Then, there is an isomorphism
\begin{equation}\label{eq:PoissSplit}
  \operatorname{Poiss}(E,\Pi)\cong Z_{\bar{\partial}^{\gamma}}^{1}\oplus\ker\rho^{\gamma}.
\end{equation}
\end{theorem}

Now, let us consider the spaces of Hamiltonian vector fields $\operatorname{Ham}(E,\Pi)$ and $\operatorname{Ham}(E,P)$ of the Poisson structures $\Pi$ and $P$, respectively. Recall that the space $B_{\bar{\partial}^{\gamma}}^{1}$ consists of $\bar{\partial}^{\gamma}$-exact 1-forms $\bar\partial^{\gamma}k$, with $k\in\operatorname{Casim}(E,P)$. Then, by using \eqref{PROP} and the fact that $P^{\sharp}(dk)=0$, we get
\[
\sharp_{H}(\bar\partial^{\gamma}k)=\Pi_{2,0}^{\sharp}(\pi^{\ast}(\bar\partial^{\gamma}k)) = \Pi_{2,0}^{\sharp}(dk)_{1,0} = \Pi_{2,0}^{\sharp}dk = \Pi^{\sharp}dk.
\]
This shows that the image of $B_{\bar{\partial}^{\gamma}}^{1}$ under the mapping $\sharp_{H}$ belongs to $\operatorname{Ham}(E,\Pi)$ and is of the form
\[
\sharp_{H}(B_{\bar{\partial}^{\gamma}}^{1})=\{\Pi_{2,0}^{\sharp}dk\mid k\in\operatorname{Casim}(E,P)\}.
\]
Furthermore, we have the following result.

\begin{proposition}
There is a short exact sequence:
\begin{equation}\label{X2}
0\rightarrow B_{\bar{\partial}^{\gamma}}^{1}\overset{\sharp_{H}}{\longrightarrow}\operatorname{Ham}(E,\Pi)\overset{\operatorname{pr}_{V}}{\longrightarrow}\operatorname{Ham}(E,P)\rightarrow0.
\end{equation}
\end{proposition}

\begin{proof}
By the nondegeneracy property of $\Pi_{2,0}^{\sharp}$, the mapping $\sharp_{H}$ is injective and hence, $\ker\sharp_{H}=\{0\}$. On the other hand, by the bigraded decomposition $\Pi=\Pi_{2,0}+P$, we conclude that $\Pi^{\sharp}df=\Pi_{2,0}^{\sharp}df+P^{\sharp}df$. This implies the equality  $\operatorname{pr}_{V}(\operatorname{Ham}(E,\Pi))=\operatorname{Ham}(E,P)$. It follows also that $\operatorname{pr}_{V}(\Pi^{\sharp}df)=P^{\sharp}df=0$ if and only if $f\in\operatorname{Casim}(E,P)$ and hence, $\ker\operatorname{pr}_{V}=\operatorname{Im}\sharp_{H}$.
\end{proof}

We observe that a necessary condition for a vector field $X$ being Hamiltonian relative to $\Pi$ and a function $f$ is that, the vertical part $X_{0,1}$ of $X$ is Hamiltonian relative to $P$ and the same function $f$. Notice also that the Poisson vector field $\sharp_{H}(\alpha)$ is Hamiltonian relative to $\Pi$ if and only if $\alpha$ is $\bar{\partial}^{\gamma}$-exact.

By \eqref{X1}, \eqref{X2}, we have the following short exact sequence of the cohomology spaces
\[
0\rightarrow\frac{Z_{\bar{\partial}^{\gamma}}^{1}}{B_{\bar{\partial}^{\gamma}}^{1}}\overset{\sharp_{H}}{\longrightarrow}\frac{\operatorname{Poiss}(E,\Pi)}{\operatorname{Ham}(E,\Pi)}\overset{\operatorname{pr}_{V}}{\longrightarrow}\frac{\ker\rho^{\gamma}}{\operatorname{Ham}(E,P)}\rightarrow0.
\]

So, we arrive at the main result.

\begin{theorem}
Let $H_{\Pi}^{1}(E)$ be the first Poisson cohomology of the coupling Poisson tensor $\Pi$ on a fiber bundle $\pi:E\rightarrow B$. Then, there exists an isomorphism
\begin{equation}\label{PC}
H_{\Pi}^{1}(E)\cong H_{\bar{\partial}^{\gamma}}^{1}\oplus\frac{\ker\rho^{\gamma}}{\operatorname{Ham}(E,P)}.
\end{equation}
\end{theorem}

By taking into account the facts in Section \ref{Sec:coupling}, as a consequence of this theorem, we derive the statement of Claim \ref{Claim:Semilocal}.

\paragraph{Regular symplectic leaves} As we have mentioned above, in the regular case, formula \eqref{PC} coincides with \eqref{F1}. Recall that the semilocal model for a Poisson structure $\Psi$ on $M$ around an embedded regular symplectic leaf $(B,\omega_{B})$ is represented by a coupling Poisson structure $\Pi$ on the normal vector bundle $\pi:E\rightarrow B$ with associated geometric data of the form $(\gamma^{0},\sigma=\omega_{B}\otimes 1+C,P=0)$ and having the zero section $B\hookrightarrow E$ as a symplectic leaf. Here,
$\gamma^{0}$ is a flat Ehresmann connection on $E$ whose horizontal distribution is just the tangent bundle $T\mathcal{S}$ of the symplectic foliation $(\mathcal{S},\omega)$ of $\Pi$. The coupling form $\sigma$ is determined by the symplectic form $\omega_{B}$ of the leaf and a
$\partial_{1,0}^{\gamma^{0}}$-closed 2-form $C\in\Omega^{2,0}_{B}(E)$ vanishing at $B$.

\begin{proposition}
We have the following relations:
\[
H^{1}_{\bar\partial^{\gamma^{0}}} \cong H^{1}_{\mathrm{dR}}(\mathcal{S}) \qquad\text{and}\qquad \ker\rho^{\gamma^{0}} = \left\{Y\in\Gamma_{\mathcal{S}\text{-}\mathrm{pr}}(\mathbb{V})\mid L_{Y}\omega\text{ is }d_{\mathcal{S}}\text{-exact}\right\}.
\]
\end{proposition}

\begin{proof}
Because of the triviality of the transverse Poisson structure of $B$, we have $\operatorname{Casim}(E,P)=C^{\infty}(E)$. Then, taking into account that $(\mathcal{C}^{\ast},\bar{\partial}^{\gamma^{0}})=(\Omega_{B}^{*,0}(E),\partial_{1,0}^{\gamma^{0}})$, we conclude that there exists a natural identification of the cochain complexes $(\Gamma(\wedge T^{\ast}\mathcal{S}),d_{\mathcal{S}})$ and $(\mathcal{C}^{\ast},\bar{\partial}^{\gamma^{0}})$. In particular, the leafwise symplectic form $\omega$ of $\Pi$ coincides with the coupling form $\sigma$. Moreover, by definition \eqref{ALP} and the relations $\operatorname{Poiss}_{V}(E,P)=\Gamma(\mathbb{V})$ and $\operatorname{Ham}(E,P)=\{0\}$, we get that
\[
\mathcal{A}^{\gamma^{0}}=\{Y\in\Gamma(\mathbb{V})\mid[\operatorname{hor}^{\gamma_{0}}(u),Y]=0\quad\forall u\in\mathfrak{X}(B)\}
\]
coincides with the space of vertical vector fields preserving the symplectic foliation, $\mathcal{A}^{\gamma^{0}}=\Gamma_{\mathcal{S}\text{-}\mathrm{pr}}(\mathbb{V})$. So, we can think of the homomorphism $\rho^{\gamma^{0}}:\mathcal{A}^{\gamma^{0}}\rightarrow H^{2}_{\bar\partial^{\gamma^{0}}}$ as a mapping $\Gamma_{\mathcal{S}\text{-}\mathrm{pr}}(\mathbb{V})\rightarrow H_{\operatorname{dR}}^{2}(\mathcal{S})$ which sends an element $Y$ to the $d_{\mathcal{S}}$-cohomology class $[L_{Y}\omega]$.
\end{proof}

In the next three sections, we discuss some other particular cases to which formula \eqref{PC} can be effectively applied.

\paragraph{Spectral Sequences} Here we briefly discuss an alternative algebraic approach to the computation of the first cohomology of the bigraded complex introduced in Section \ref{Sec:coupling} (see also \cite{Va-90,Va-94}, for the use of spectral sequences in the computation of Poisson cohomology).

Consider the nonnegative cochain complex $(\mathfrak{M}^{\ast}=\bigoplus_{n\in\mathbb{Z}}\mathfrak{M}^{n},\partial)$, where
\[
\mathfrak{M}^{n}=\bigoplus_{p+q=n}\Omega_{B}^{p,q}(E)
\]
and the differential operator $\partial$ is given by \eqref{Co}. Consider a filtration $\text{\emph{F}}$ of $\mathfrak{M}$ defined as
\[
\text{\emph{F}}^{p}\mathfrak{M}^{n}:=\bigoplus_{p\leq k\leq n}\Omega_{B}^{k,n-k}(E).
\]
Then,
\[
\mathfrak{M}^{n}=\text{\emph{F}}^{0}\mathfrak{M}^{n}\supset\text{\emph{F}}^{1}\mathfrak{M}^{n}\supset\cdots\supset\text{\emph{F}}^{n}\mathfrak{M}^{n}=\Omega_{B}^{n,0}(E)\supset\text{\emph{F}}^{n+1}\mathfrak{M}^{n}=\{0\}
\]
and hence the filtration $\text{\emph{F}}$ is \emph{bounded} \cite{McCleary}. Moreover, the bigraded decomposition \eqref{Co} provides the inclusions $\partial(\text{\emph{F}}^{p}\mathfrak{M}^{n})\subseteq\text{\emph{F}}^{p}\mathfrak{M}^{n+1}$ for all $p$ and $n$. Therefore, $(\mathfrak{M},\partial,\text{\emph{F}})$ is a graded filtered complex.

Now, consider the spectral sequence $(E_{r},d_{r})_{r\geq0}$ associated with $(\mathfrak{M},\partial,\text{\emph{F}})$. Observe that $(E_{r},d_{r})_{r\geq0}$ is a first quadrant spectral sequence.

\begin{lemma}\label{lemma:SpecSeq}
The spectral sequence $(E_{r},d_{r})_{r\geq0}$ converges to the cohomology of the cochain complex $(\mathfrak{M},\partial)$. Moreover, we have
\begin{equation}\label{eq:SplitSpec1}
  H^{1}(\mathfrak{M},\partial) \cong E_{2}^{1,0}\oplus E_{3}^{0,1}.
\end{equation}
\end{lemma}

\begin{proof}
  Since $(E_{r},d_{r})_{r\geq0}$ is a first quadrant spectral sequence, for every $p,q\in\mathbb{Z}$ we have $E^{p,q}_{\infty}=E^{p,q}_{N}$, where $N=\max\{p+1,q+2\}$. Then, the convergence and relation \eqref{eq:SplitSpec1} follow from the fact that $\text{\emph{F}}$ is bounded (see \cite{McCleary}).
\end{proof}

To compute $E_{2}^{1,0}$ and $E_{3}^{0,1}$ in the notations introduced in Section \ref{Sec:coupling},  we use the explicit general formulas for the $E$-terms of a spectral sequence (see, for example, \cite{DZ}):
\begin{equation*}
  E_{r}^{p,q} = \frac{\ker(\partial|_{\mathfrak{M}^{p+q}})\cap\text{\emph{F}}^{p}\mathfrak{M}^{p+q}+\text{\emph{F}}^{p+1}\mathfrak{M}^{p+q}}{\operatorname{Im}(\partial|_{\mathfrak{M}^{p+q-1}})\cap\text{\emph{F}}^{p}\mathfrak{M}^{p+q}+\text{\emph{F}}^{p+1}\mathfrak{M}^{p+q}}, \qquad r\geq\max\{p+1,q+2\}.
\end{equation*}
The direct computations give
\begin{align*}
E_{2}^{1,0} &= \frac{\ker(\partial|_{\mathfrak{M}^{1,0}})}{\operatorname{Im}(\partial^{\gamma}_{1,0}|_{\ker\partial^{P}_{0,1}\cap\mathfrak{M}^{0}})} = \frac{Z^{1}_{\bar{\partial}^{\gamma}}}{B^{1}_{\bar{\partial}^{\gamma}}} = H^{1}_{\bar{\partial}^{\gamma}},
\\
E_{3}^{0,1} &= \frac{\ker(\partial|_{\mathfrak{M}^{1}})+\mathfrak{M}^{1,0}}{\operatorname{Im}(\partial|_{\mathfrak{M}^{0}})+\mathfrak{M}^{1,0}} =\frac{\operatorname{pr}_{V}(\ker(\partial|_{\mathfrak{M}^{1}}))\oplus\mathfrak{M}^{1,0}}{\operatorname{pr}_{V}(\operatorname{Im}(\partial|_{\mathfrak{M}^{0}}))\oplus\mathfrak{M}^{1,0}} \cong\frac{\ker\rho^{\gamma}}{\operatorname{Ham}(E,P)},
\end{align*}
where we use the relations: $\operatorname{pr}_{V}(\ker(\partial|_{\mathfrak{M}^{1}}))=\ker\rho^{\gamma}$ and $\operatorname{pr}_{V}(\operatorname{Im}(\partial|_{\mathfrak{M}^{0}}))=\operatorname{Ham}(E,P)$ (see Lemma \ref{lemma:KerRho}). This shows that formula \eqref{eq:SplitSpec1} coincides with \eqref{PC} under the cochain complex isomorphism $\flat_{\sigma}$ (see the proof of Proposition \ref{prop:IsoCompx}).

\section{Vanishing of the First Poisson Cohomology}\label{Sec:Vanish}

Here, by using the results of the previous section, we present some sufficient conditions for the vanishing of the first Poisson cohomology.

Let $\pi:E\rightarrow B$ be a fiber bundle. Suppose that we start again with a coupling Poisson tensor $\Pi$ on $E$ with associated geometric data $(\gamma,\sigma,P)$. We make the following assumptions. Assume that the \emph{first vertical cohomology group} of $P$ is trivial, that is,
\begin{equation}\label{PC1}
\operatorname{Poiss}_{V}(E,P)=\operatorname{Ham}(E,P).
\end{equation}

It follows from \eqref{CC2} and \eqref{PC1} that the horizontal lifts of every $u\in\mathfrak{X}(B)$ with respect to two Poisson connections on $(E,P)$ differ by a Hamiltonian vector field. Then, by \eqref{Def} and definition \eqref{DEF}, we conclude that the coboundary operator $\bar{\partial}^{\gamma}$ is independent of the choice of the Poisson connection $\gamma$ on $(E\xrightarrow{\pi}B,P)$ and will be simply denoted by $\bar{\partial}$. Therefore, under condition \eqref{PC1}, one can associate to the Poisson bundle the intrinsic cochain complex $(\mathcal{C}^{*}=\oplus_{p}\mathcal{C}^{p},\bar{\partial})$. Taking into account \eqref{PC1} and \eqref{IN}, we conclude that $\ker\rho^{\gamma}=\operatorname{Ham}(E,P)$ and hence,
\[
\operatorname{Poiss}(E,\Pi)\cong Z_{\bar{\partial}}^{1}\oplus\operatorname{Ham}(E,P).
\]
So, in this case, formula \eqref{PC} for the first Poisson cohomology of $\Pi$ reads
\begin{equation}\label{PC5}
H_{\Pi}^{1}(E)\cong H_{\bar{\partial}}^{1}.
\end{equation}

Next, let us assume that $(E,P)$ is a \emph{flat Poisson bundle}, that is,
\begin{equation}\label{PC4}
\text{there exists a \emph{flat Poisson connection} } \gamma^{0}\text{ on }(E,P).
\end{equation}
Equivalently, condition \eqref{PC4} can be reformulated as follows: there exists a regular foliation $\mathcal{F}$ on $E$ such that
\begin{equation}\label{PC6}
TE=T\mathcal{F\oplus}\mathbb{V}\quad(\mathbb{V}:=\ker d\pi)
\end{equation}
and every $\pi$-projectable section of $T\mathcal{F}$ is a Poisson vector field on $(E,P)$,
\begin{equation}\label{PC7}
L_{Z}P=0\qquad\forall Z\in\Gamma_{\pi\text{-}\mathrm{pr}}(T\mathcal{F}).
\end{equation}
Then, the horizontal subbundle $\mathbb{H}^{\gamma^{0}}$ of the flat Poisson connection $\gamma^{0}$ is just the tangent bundle $T\mathcal{F}$ of the foliation. Recall that $\Gamma_{\pi\text{-}\mathrm{pr}}(T\mathcal{F}) = \{\operatorname{hor}^{\gamma^{0}}(u)\mid u\in\mathfrak{X}(B)\}$ denotes the space of all $\pi$-projectable, $\gamma^{0}$-horizontal vector fields on $E$.

Let $\Omega^{p}(\mathcal{F}):=\Gamma(\wedge^{p}(T\mathcal{F})^{\ast})$ be the space of foliated $p$-forms on $E$. In particular, $\Omega^{0}(\mathcal{F})=C^{\infty}(E)$. Consider the \emph{foliated de Rham complex} $(\Omega^{\ast}(\mathcal{F})=\bigoplus_{p\in\mathbb{Z}}\Omega^{p}(\mathcal{F}),d_{\mathcal{F}})$, where $d_{\mathcal{F}}:\Omega^{p}(\mathcal{F})\rightarrow\Omega^{p+1}(\mathcal{F})$ is the \emph{foliated exterior differential} given by the standard formula
\begin{align*}
(d_{\mathcal{F}}\mu)(X_{0},\ldots,X_{p}):=&\sum_{i=0}^{p}(-1)^{i}L_{X_{i}}(\mu(X_{0},\ldots,\hat{X}_{i},\dots,X_{p}))\\
&+\sum_{i<j} (-1)^{i+j}\mu([X_{i},X_{j}],X_{0},\ldots,\hat{X}_{i},\ldots,\hat{X}_{j},\ldots,X_{p}).\nonumber
\end{align*}
The cohomology of $(\Omega^{*}(\mathcal{F}),d_{\mathcal{F}})$ is called the \emph{foliated de Rham cohomology} and denoted by $H_{\operatorname{dR}}^{*}(\mathcal{F})$. Observe that $(\Omega^{*}(\mathcal{F}),d_{\mathcal{F}})$ is isomorphic to the cochain complex $(\Gamma(\wedge^{*}\mathbb{V}^{0}),d^{\gamma^{0}}_{1,0})$, where $d^{\gamma^{0}}_{1,0}$ is the component of bidegree $(1,0)$ of the exterior differential relative to the flat connection $\gamma^{0}$ \cite{Va-04}. More precisely,
\[
d^{\gamma^{0}}_{1,0}\beta(Y_{0},\ldots,Y_{p}):=d\beta(p_{T\mathcal{F}}Y_{0},\ldots,p_{T\mathcal{F}}Y_{p})
\]
for any $\beta\in\Gamma(\wedge^{p}\mathbb{V}^{0})$ and vector fields $Y_{0},\ldots,Y_{p}$ on $E$. Here, $p_{T\mathcal{F}}:TE\rightarrow T\mathcal{F}$ is the projection along $\mathbb{V}$. It follows from \eqref{Cur2} and \eqref{PROP} that $(\Omega^{*,0}_{B}(E),\partial_{1,0}^{\gamma^{0}})$ is a cochain complex is isomorphic to $(\Gamma(\wedge^{*}\mathbb{V}^{0}),d^{\gamma^{0}}_{1,0})$.

Now let us associate to the flat Poisson bundle $(E\overset{\pi}{\rightarrow}B,P,\mathcal{F})$ the following subalgebra of Hamiltonian vector fields:
\[
\operatorname{Ham}_{\mathcal{F}}(E,P):=\{Y\in\operatorname{Ham}(E,P)\mid[Y,\Gamma_{\pi\text{-}\mathrm{pr}}(T\mathcal{F})]=0\}.
\]
Observe that a Hamiltonian vector field $Y$ on $(E,P)$ belongs to $\operatorname{Ham}_{\mathcal{F}}(E,P)$ if and only if the flow of
$Y$ preserves the foliation $\mathcal{F}$. \ Let $C_{\mathcal{F}}^{\infty}(E):=\{H\in C^{\infty}(E)\mid d_{\mathcal{F}}H=0\}$ be the space of smooth functions on $E$ which are constant along the leaves of $\mathcal{F}$.

\begin{theorem}\label{teo:Vanish1}
Suppose that in addition to hypotheses \eqref{PC1}, \eqref{PC4}, the following conditions hold:
\begin{enumerate}
\item[(i)] the first foliated de Rham cohomology group of $(E,\mathcal{F})$ is trivial,
\begin{equation}\label{A2}
H_{\operatorname{dR}}^{1}(\mathcal{F})=\{0\};
\end{equation}
\item[(ii)] the subalgebra of Hamiltonian vector fields on $(E,P)$ preserving the foliation $\mathcal{F}$ is generated by the subspace $C_{\mathcal{F}}^{\infty}(E)$,
\begin{equation}\label{A4}
\operatorname{Ham}_{\mathcal{F}}(E,P) = \{P^{\sharp}dH\mid H\in C_{\mathcal{F}}^{\infty}(E)\}.
\end{equation}
\end{enumerate}
Then, the first Poisson cohomology of the coupling Poisson tensor \ $\Pi$ vanishes,
\begin{equation}\label{TC}
H_{\Pi}^{1}(E)=\{0\}.
\end{equation}
\end{theorem}

\begin{proof}
First, we observe that condition \eqref{A4} can be reformulated as follows: For every function $f\in C_{E}^{\infty}$ with property:
\begin{equation}\label{A3}
L_{Z}f\in\operatorname{Casim}(E,P)\quad\forall Z\in\Gamma_{\pi\text{-}\mathrm{pr}}(T\mathcal{F}),
\end{equation}
there exists $g\in\operatorname{Casim}(E,P)$ such that
\begin{equation}\label{A1}
d_{\mathcal{F}}f=d_{\mathcal{F}}g.
\end{equation}
Indeed, it follows from \eqref{A3}, \eqref{A1} that $[P^{\sharp}df,Z]=-P^{\sharp}dL_{Z}f=0$ and $P^{\sharp}df=P^{\sharp}dH$, where $H=f-g\in$
$C_{\mathcal{F}}^{\infty}(E)$. Now, let us define
\[
\bar{\Omega}^{q}(\mathcal{F}):=\left\{\beta\in\Omega^{q}(\mathcal{F}) \mid  \mathbf{i}_{X_{1}}\ldots\mathbf{i}_{X_{q}}\beta\in\operatorname{Casim}(E,P) \quad\forall X_{i}\in\Gamma_{\pi\text{-}\mathrm{pr}}(T\mathcal{F})\right\}.
\]
In particular, $\bar{\Omega}^{0}(\mathcal{F})=\operatorname{Casim}(E,P)$. Using the property that $\Gamma_{\pi\text{-}\mathrm{pr}}(T\mathcal{F})\subset\operatorname{Poiss}(E,P)$, it easy to see that the foliated differential $d_{\mathcal{F}}$ leaves invariant the subspaces $\bar{\Omega}^{q}(\mathcal{F})$ of $\Omega^{q}(\mathcal{F})$ and hence, the coboundary operator $\bar{d}_{\mathcal{F}}:=d_{\mathcal{F}}|_{\bar{\Omega}^{q}(\mathcal{F})}$ is well defined. Then,  $(\bar{\Omega}^{\ast}(\mathcal{F}):=\oplus_{q}\bar{\Omega}^{q}(\mathcal{F}),\bar{d}_{\mathcal{F}})$ is a subcomplex of the cochain complex $(\Omega^{*}(\mathcal{F}),d_{\mathcal{F}})$ and there is a natural homomorphism $H_{\bar{d}_{\mathcal{F}}}^{q}\rightarrow H_{\operatorname{dR}}^{q}(\mathcal{F})$ between the corresponding cohomology groups. One can show that conditions \eqref{A3}, \eqref{A1} are equivalent to the following:
\[
d_{\mathcal{F}}(\bar{\Omega}^{0}(\mathcal{F})) = d_{\mathcal{F}}(\Omega^{0}(\mathcal{F}))\cap\bar{\Omega}^{1}(\mathcal{F}).
\]
This condition means that the natural homomorphism $H_{\bar{d}_{\mathcal{F}}}^{1}\rightarrow H_{\operatorname{dR}}^{1}(\mathcal{F})$ is injective. Therefore, the hypotheses (i), (ii) of the theorem imply that $H_{\bar{d}_{\mathcal{F}}}^{1}=\{0\}$. Finally, we observe that the cochain complexes associated with $\bar{d}_{\mathcal{F}}$ and $\bar{\partial}:=\partial_{1,0}^{\gamma^{0}}|_{\mathcal{C}^{\ast}}$ are isomorphic and hence, $H_{\bar{\partial}}^{1}$ is trivial. This, together with \eqref{PC5}, proves \eqref{TC}.
\end{proof}

To get more insight for the criterion in Theorem \ref{teo:Vanish1}, let us consider the situation when conditions \eqref{PC1}, \eqref{PC4} are fulfilled and the foliation $\mathcal{F}$ is a fibration. In other words, we assume that the leaf space $K:=E\diagup\mathcal{F}$ of the foliation is a smooth manifold and the natural projection $\nu:E\rightarrow K$ is a surjective submersion. So, we have $T\mathcal{F}=\ker d\nu$.

\begin{lemma}\label{lemma:ExistsPoiss}
There exists a Poisson structure $\Upsilon$ on $K$ such that the projection $\nu:E\rightarrow K$ is a Poisson map. Moreover, condition \eqref{A4} is equivalent to the following property: the Hamiltonian vector field $P^{\sharp}df$ of every function $f\in C^{\infty}(E)$ such that
\begin{equation}\label{HH1}
[Z,P^{\sharp}df]=0\quad\forall Z\in\Gamma_{\pi\text{-}\mathrm{pr}}(T\mathcal{F}),
\end{equation}
is $\nu$-related with a Hamiltonian vector field on $(K,\Upsilon)$,
\begin{equation}\label{HH2}
d\nu\circ P^{\sharp}df=\Upsilon^{\sharp}dh\circ\nu,
\end{equation}
for a certain $h\in C^{\infty}(K)$.
\end{lemma}

\begin{proof}
Notice that $C_{\mathcal{F}}^{\infty}(E)=\nu^{\ast}C^{\infty}(K)$. This and condition \eqref{PC7} imply that
\[
L_{Z}(P(d\nu^{\ast}\kappa_{1},d\nu^{\ast}\kappa_{2}))= P(dL_{Z}(\nu^{\ast}\kappa_{1}),d\nu^{\ast}\kappa_{2})+P(d\nu^{\ast}\kappa_{1},dL_{Z}(\nu^{\ast}\kappa_{2}))=0
\]
for any $Z\in\Gamma_{\pi\text{-}\mathrm{pr}}(T\mathcal{F})$ and $\kappa_{1}, \kappa_{2}\in C^{\infty}(K)$. It follows that there exists a bivector field $\Upsilon\in\mathfrak{X}^{2}(K)$ which is uniquely determined by
\[
\nu^{\ast}\Upsilon(d\kappa_{1},d\kappa_{2})=P(d\nu^{\ast}\kappa_{1},d\nu^{\ast}\kappa_{2})
\]
and satisfies the Jacobi identity. Now, condition \eqref{HH1} for $f\in C^{\infty}(E)$ says that $P^{\sharp}df\in\operatorname{Ham}_{\mathcal{F}}(E,P)$ and hence \eqref{A3} holds. It remains to show the equivalence of conditions \eqref{A1} and \eqref{HH2}. Indeed, by \eqref{HH1} the Hamiltonian vector field $P^{\sharp}df$ is $\nu$-related with a vector field $w\in\mathfrak{X}(K)$ which is an infinitesimal automorphism of $\Upsilon$. Then, it is easy to see that $w$ is Hamiltonian, $w=\Upsilon^{\sharp}dh,$ $h\in C^{\infty}(K)$ if and only if $f$ satisfies \eqref{A1}, where $g=f-\nu^{\ast}h\in\operatorname{Casim}(E,P)$.
\end{proof}

Observe that the hypotheses (i) and (ii) of Theorem \ref{teo:Vanish1} are independent in general. Here is a particular case in which condition \eqref{HH2} is satisfied but \eqref{PC1} or \eqref{A2} do not necessarily hold.

\begin{example}
Let $B$ be a manifold and consider a Poisson manifold $K$ equipped with a Poisson tensor $\Upsilon$. Let $(E=B\times K, P)$ be the product of  Poisson manifolds, where $B$ has the trivial Poisson structure. Let us think of $(E,P)$ as the total space of a trivial Poisson bundle over $B$ with projection $\pi=\operatorname{pr}_{1}$ and the vertical Poisson structure $P$. It is clear that $P$ and $\Upsilon$ are $\operatorname{pr}_{2}$-related and $\ker(\operatorname{pr}_{2})\subset TE$ induces a flat Poisson connection for $P$. Fixing $x_{0}\in B$, consider the section $s:K\rightarrow E$ of $\nu$ given by $s(y)=(x_{0},y)$. Pick a function $f\in C^{\infty}(E)$ satisfying \eqref{HH1} and put $h=s^{\ast}f$. Then, one can easily verify that \eqref{HH2} holds.
\end{example}

By the same arguments as in the proof of Lemma \ref{lemma:ExistsPoiss}, we derive the following cohomological criterion.

\begin{lemma}\label{lemma:Hups0}
Under hypotheses \eqref{PC1}, \eqref{PC4}, in the case when the foliation $\mathcal{F}$ is a fibration, condition \eqref{A4} is equivalent to the triviality of the first Poisson cohomology group of $\Upsilon$.
\end{lemma}

\begin{proof}
First assume that $H_{\Upsilon}^{1}(K)=\{0\}$ and let $f\in C_{E}^{\infty}$ be such that $P^{\sharp}df$ preserves the foliation $\mathcal{F}$. As in the proof of Lemma \ref{lemma:ExistsPoiss}, $P^{\sharp}df$ is $\nu$-related to some infinitesimal automorphism $w\in\mathfrak{X}(K)$ of $\Upsilon$. By hypothesis, $w=\Upsilon^{\sharp}dh$, so $P^{\sharp}d(\nu^{\ast}h)=P^{\sharp}df$, by the uniqueness of the horizontal lift of  $w\in\mathfrak{X}(K)$ in the fibration $\nu:E\rightarrow K$ with horizontal distribution $\ker d\pi$. Therefore, condition \eqref{A4} holds. Conversely, let $w\in\mathfrak{X}(K)$ be an infinitesimal automorphism of $\Upsilon$. If $W\in\mathfrak{X}(E)$ is the horizontal lift of $w$
as described in above, then $[W,P]\in\Gamma(\mathbb{V)}\cap\Gamma(\ker d\nu)=\{0\}$, so $W\in\operatorname{Poiss}_{V}(E,P)$. By \eqref{PC1}, $W$ is
Hamiltonian. Furthermore, it follows from \eqref{A4} that, $W=P^{\sharp}d(\nu^{\ast}h)$ for some $h\in C^{\infty}(K)$. Hence, $w=\Upsilon^{\sharp}dh$.
\end{proof}

Summarizing the above considerations, we arrive at the following result.

\begin{theorem}\label{Teo:BiFib}
Let $K\overset{\nu}{\longleftarrow}E\overset{\pi}{\longrightarrow}B$ be a transversal bi-fibration, that is, $\nu$ and $\pi$ are surjective submersions and
\begin{equation}\label{eq:BiFib}
TE = \ker d\nu\mathcal{\oplus}\ker d\pi.
\end{equation}
Suppose that $\nu$ has connected fibers and satisfies the following compatibility condition with a Poisson tensor $P\in\Gamma(\wedge^{2}\ker d\pi)$:
\begin{equation}\label{ZX0}
\Gamma_{\pi\text{-}\mathrm{pr}}(\ker d\nu)\subset\operatorname{Poiss}(E,P).
\end{equation}
Let $\mathcal{F}$ be the regular foliation on $E$ with $T\mathcal{F}=\ker d\nu$ and $\Upsilon$ be a unique Poisson structure on $K$ for which the natural projection $\nu:E\rightarrow K$ is a Poisson map. Suppose we are given a coupling Poisson structure $\Pi=\Pi_{2,0}+\Pi_{0,2}$ on the fiber
bundle $\pi:E\rightarrow B$ with vertical part $\Pi_{0,2}=P$. Then, the first Poisson cohomology group of $\Pi$ on $E$ is trivial if
\begin{align}
\frac{\operatorname{Poiss}_{V}(E,P)}{\operatorname{Ham}(E,P)}&=\{0\},\label{ZX1}\\
H_{\operatorname{dR}}^{1}(\mathcal{F})&=\{0\}, \label{ZX2}\\
H_{\Upsilon}^{1}(K)&=\{0\}. \label{ZX3}
\end{align}
\end{theorem}

We end this section with the following remarks about conditions \eqref{ZX0}-\eqref{ZX3}. First we notice that if the $\nu$-fibers are simply connected, then condition \eqref{ZX2} holds (see, for example, \cite{DazHec-91}).

Moreover, by \eqref{eq:BiFib}, the restriction of the surjective submersion $\nu$ to each $\pi$-fiber  $\nu|_{\pi^{-1}(x)}:(\pi^{-1}(x),P_{x})\rightarrow(K,\Upsilon)$ is a local Poisson diffeomorphism. We claim that conditions \eqref{ZX0} and \eqref{ZX1} imply \eqref{ZX3} if the restriction $\nu|_{E_{0}}:E_{0}\rightarrow K$ is bijective for a single fiber $E_{0}:=\pi^{-1}(x_{0})$.
This fact is based on the following observation: the lifting $W\in\Gamma_{\nu\text{-}\operatorname{pr}}(\mathbb{V})$ of every element $w\in\operatorname{Poiss}(K,\Upsilon)$ is a Poisson vector field on $(E,P)$ (see Lemma \ref{lemma:Hups0}). 

\section{Isotropy Algebras of Compact Semisimple type}\label{Sec:Isotropy}

The triviality condition (\ref{PC1}) is realized in the following case. Let $(E\overset{\pi}{\rightarrow}B,P)$ be a locally trivial Lie-Poisson bundle whose typical fiber is the co-algebra $(\mathfrak{g}^{\ast},\Lambda)$ of a semisimple Lie algebra $\mathfrak{g}$ of compact type. Recall that this condition means that the Killing form is negative definite or, equivalently, that the connected and simply connected Lie group integrating $\mathfrak{g}$ is compact. Due to \cite{Conn-85} (see, also \cite{Marcut-13}), we have $H_{\Lambda}^{1}(\mathfrak{g}^{\ast})=0$. Moreover, there exist the linear homotopy operators for the Poisson complex of $(\mathfrak{g}^{\ast},\delta_{\Lambda})$ in degree 1:
\[
C^{\infty}(\mathfrak{g}^{\ast})\overset{h_{0}}{\longleftarrow}\mathfrak{X}(\mathfrak{g}^{\ast})\overset{h_{1}}{\longleftarrow}\mathfrak{X}^{2}(\mathfrak{g}^{\ast}),
\]
\[
\delta_{\Lambda}\circ h_{0} + h_{1}\circ\delta_{\Lambda} = \operatorname{Id}_{\mathfrak{X}(\mathfrak{g}^{\ast})}.
\]
Observe that this fact remains true if instead of $\mathfrak{g}^{\ast}$ we take an open ball (with respect to the invariant inner product in  $\mathfrak{g}^{\ast})$ centered at the origin. The existence of the homotopy operators imply the triviality of the parametrized first Poisson cohomology groups of the Lie-Poisson structure $\Lambda$. Combining this fact with the partition unity argument, we conclude that the first vertical
cohomology group of $P$ is also trivial.

\begin{remark}
The triviality property of the parametrized first cohomology groups appears also in the context of the \emph{tame Poisson structures}, introduced in \cite{MZ-06}.
\end{remark}

Now, as an illustration of Theorem \ref{Teo:BiFib}, let us consider the following situation. Let $M=B\times\mathbb{R}^{k}$ be the product of a compact connected symplectic manifold $B$ and the $k$-dimensional Euclidean space $\mathbb{R}^{k}=\{x=(x^{1},\ldots,x^{k})\}$. Let us view $M$ as
the total space of the trivial vector bundle over $B$. Suppose we are given a Poisson tensor $\Pi$ on $M$ such that the zero section $B\times\{0\}$ is a symplectic leaf of $\Pi$. Assume that
\begin{equation}\label{SC}
c_{\sigma}^{\alpha\beta} := \left.\frac{\partial}{\partial x^{\sigma}}\Pi(dx^{\alpha},dx^{\beta})\right|_{x=0} = \operatorname{const} \text{ on } B.
\end{equation}
Then, $c_{\sigma}^{\alpha\beta}$ are the structure constants of a Lie algebra $\mathfrak{g}$ and condition \eqref{SC} means that the isotropy bundle of the leaf $B\times\{0\}$ is just the trivial Lie bundle $B\times\mathfrak{g}$. Observe that after a change of coordinates on the fiber, condition \eqref{SC} still holds. Combining Theorem \ref{Teo:BiFib} with Conn's results \cite{Conn-85}, we establish the following criterion which implies Claim \ref{Claim:LieBundle}.

\begin{proposition}\label{prop:SympLeaf}
If $B$ is simply connected and compact and the isotropy algebra $\mathfrak{g}$ is semisimple of compact type, then there exists an open neighborhood $E$ of $B\times\{0\}$ in $M=B\times\mathbb{R}^{k}$ such that \ $H_{\Pi}^{1}(E)=\{0\}$.
\end{proposition}

\begin{proof}
We have to verify that the hypotheses of the proposition imply conditions \eqref{ZX1}-\eqref{ZX3}. First, we observe that $\Pi=\Pi_{2,0}+\Pi_{0,2}$ is a coupling Poisson structure in a neighborhood $E$ of $B\times\{0\}$ in $M=B\times\mathbb{R}^{k}$ which is viewed as \ the total space of the fiber bundle $\pi:=\operatorname{pr}_{1}|_{E}$ over $B$. \ By \eqref{SC}, the linearization of the vertical Poisson structure \ $P=\Pi_{0,2}$ at
$B\times\{0\}$ gives
\begin{equation}\label{LL}
P^{(1)} = \frac{1}{2}c_{\sigma}^{\alpha\beta}x^{\sigma}\frac{\partial}{\partial x^{\alpha}}\wedge\frac{\partial}{\partial x^{\beta}}.
\end{equation}
By the linearization Conn theorem, for every $b\in B$, the Poisson structure $P_{b}$ on the fiber $E_{b}$ around $0$ is isomorphic to the Lie-Poisson structure $\Lambda$ on $\mathfrak{g}^{\ast}$. Then, one can show \cite{Vo-05} that the neighborhood $E$ can be chosen in such a way that there exists a fiber preserving diffeomorphism $g:E\rightarrow g(E)$ identical on $B$ and $g_{\ast}P=P^{(1)}.$ \ So, we obtain the coupling Poisson tensor $g_{\ast}\Pi=$ $g_{\ast}\Pi_{2,0}+P^{(1)}$ defined on the neighborhood $g(E)$ of $B$. \ Then, as we mentioned above, condition \eqref{ZX1} holds for $P^{(1)}$. \ Moreover, by the compactness of $B$, one can arrange the neighborhood $E$ to have $g(E)=B\times K$, where $K$ is an open ball centered at $0$. Then, condition \eqref{ZX3} holds for $\Upsilon=\Lambda$, $H_{\Lambda}^{1}(K)=\{0\}$. Finally, by condition \eqref{SC}, there exists a flat Poisson connection $\gamma^{0}$ on $(B\times K,P^{(1)})$ associated with the horizontal foliation $\mathcal{F}$ with leaves  $B\times\{x\},$ $x\in K$. Then, the foliated de Rham cohomology of $d_{\mathcal{F}}$ \ is the same thing as the de Rham cohomology of the forms in $B$ depending smoothly on $x\in K$ as a parameter. \ Since $B$ is simply connected, according to results in \cite{DazHec-91,GLSW}, we conclude that
$H_{\operatorname{dR}}^{1}(\mathcal{F})=\{0\}$.
\end{proof}

\begin{example}
Consider the case when $B=\mathbb{S}^{2}\subset\mathbb{R}^{3}$ is the unit 2-sphere equipped with the area form $\omega=dp\wedge
dq$. \ Here, the Darboux coordinates $p,q$ can be defined as the azimuthal angle $p=\varphi$ and the height function $q=h$ on the sphere. Then, given a vector valued 1-form $\mathbf{\varrho}$ on $B$,
\[
\varrho=\varrho^{(1)}(p,q,x)dp+\varrho^{(2)}(p,q,x)dq,
\]
with $\varrho^{(1)},\varrho^{(2)}\in\mathbb{R}^{3}$, and a constant $c\in\mathbb{R}$, one can define the following Poisson tensor on $M=\mathbb{S}^{2}\times\mathbb{R}^{3}$ \cite{Vo-05}:
\begin{align*}
\Pi_{\varrho,c}  &= \frac{1}{2\left(1 - \Delta_{\varrho} + c\|x\|^{2}\right)} \left(\frac{\partial}{\partial p}+(x\times\varrho^{(1)}) \cdot \frac{\partial}{\partial x}\right) \wedge \left(\frac{\partial}{\partial q} + (x\times\varrho^{(2)})\cdot\frac{\partial}{\partial x}\right) \\
&+\frac{1}{2}\epsilon_{\alpha\beta\gamma}x^{\gamma}\frac{\partial}{\partial x^{\alpha}}\wedge\frac{\partial}{\partial x^{\beta}},
\end{align*}
where $\Delta_{\varrho}:=\frac{\partial\varrho^{(2)}}{\partial p} - \frac{\partial\varrho^{(1)}}{\partial q} + x\cdot\frac{\partial\varrho^{(1)}}{\partial x} \times\frac{\partial\varrho^{(2)}}{\partial x}$. In this case, $\mathbb{S}^{2}\times\{0\}$ is a simply connected, compact symplectic leaf of $\Pi_{\varrho,c}$ whose isotropy Lie algebra is $\mathfrak{g}=\operatorname{so}(3)$. Therefore, by Proposition \ref{prop:SympLeaf} we conclude that the first cohomology of $\Pi_{\varrho,c}$ vanishes for arbitrary data $(\varrho,c)$. In
particular, this is true for the product Poisson structure $\Pi_{0,0}$ on $\mathbb{S}^{2}\times\operatorname{so}^{\ast}(3)$.
\end{example}

\section{Projectability of Casimir Functions}\label{Sec:ProyCasim}

Let $(E\overset{\pi}{\rightarrow}B,P)$ be again a Poisson bundle and $\Pi$ a coupling Poisson structure on $E$ with associated geometric data $(\gamma,\sigma,P)$. Let us consider another extreme situation, assuming that every Casimir function of the vertical Poisson structure $P$ is \emph{projectable} in the sense that
\begin{equation}\label{PC2}
\operatorname{Casim}(E,P)=\pi^{\ast}C^{\infty}(B).
\end{equation}
So, this means that $P^{\sharp}dF=0$ if and only if $F=\pi^{\ast}f$ for a certain $f\in C^{\infty}(B)$.

\begin{example}\label{ex:OpenBook}
Let
\begin{equation}\label{LA2}
\Lambda = \frac{\partial}{\partial x_{1}} \wedge \left(x_{2}\frac{\partial}{\partial x_{2}} + x_{3}\frac{\partial}{\partial x_{3}}\right)
\end{equation}
be the Lie-Poisson structure on the co-algebra $\mathfrak{g}^{\ast}=\mathbb{R}^{3}$ of the 3-dimensional Lie algebra
\begin{equation}\label{LA}
[e_{1},e_{2}]=e_{2},\quad[e_{2},e_{3}]=0,\quad[e_{3},e_{1}]=-e_{3}.
\end{equation}
In this case, the foliation of $\mathbb{R}^{3}$ by the symplectic leaves (the co-adjoint orbits) is an open book type foliation. As a consequence, the corresponding Lie-Poisson structure $\Lambda$ on $\mathfrak{g}^{\ast}$ does not admit any global nontrivial Casimir function on $\mathbb{R}^{3}$, that is, $\operatorname{Casim}(\Lambda,\mathbb{R}^{3})=\mathbb{R}$. One can show that the first cohomology group of the Lie-Poisson structure $\Lambda$ in \eqref{LA2} is generated by the Poisson vector fields
\begin{equation}\label{eq:BasisZ}
Z_{1}=\frac{\partial}{\partial x_{1}},
Z_{2}=x_{2}\frac{\partial}{\partial x_{2}}-x_{3}\frac{\partial}{\partial x_{3}},
Z_{3}=x_{3}\frac{\partial}{\partial x_{2}},
Z_{4}=x_{2}\frac{\partial}{\partial x_{3}}
\end{equation}
and, hence, isomorphic to $\mathbb{R}\times\mathfrak{sl}(2,\mathbb{R})$ as Lie algebras.
\end{example}

It follows that condition \eqref{PC2} holds for any locally trivial Lie-Poisson bundle $(\pi:E\rightarrow B,P)$ over $B$ whose typical fiber is just $\mathbb{R}^{3}$ equipped with linear Poisson bracket \eqref{LA2}.

\begin{remark}
The fact that $H^{1}_{\Lambda}(\mathbb{R}^{3})$ is generated by the basis \eqref{eq:BasisZ} can be stated by direct computations. It is of interest to note that in the regular domain $\mathbb{R}^{3}\backslash\{x_{1}\text{-axis}\}\cong\mathbb{R}^{2}\times\mathbb{S}^{1}$, the Poisson structure \eqref{LA2} has nontrivial Casimir functions and, as a consequence, the first Poisson cohomology group is infinite dimensional and isomorphic to $C^{\infty}(\mathbb{S}^{1})$ \cite{AG}. More examples of explicit computations of the first cohomology of low-dimensional Poisson manifolds with singularities can be also found in \cite{Mon-02,Nak-95}.
\end{remark}

Now, we observe that the condition \eqref{PC2} implies that the cochain complex $(\mathcal{C}^{*},\bar\partial^{\gamma})$ is isomorphic to the de Rham complex $(\Omega^{*}(B),d_{B})$ on the base $B$. Therefore, in this case we have $H_{\bar{\partial}^{\gamma}}^{k}\cong H_{\operatorname{dR}}^{k}(B)$ $\forall k\geq0$.

\begin{proposition}\label{prop:CasimCase0}
  If, in addition to \eqref{PC2}, the second de Rham cohomology of the base $B$ is trivial,
  \begin{equation}\label{PCF2}
  H_{\operatorname{dR}}^{2}(B)=\{0\},
  \end{equation}
  then
  \begin{equation}\label{PCF9}
  \operatorname{Poiss}(E,\Pi)\cong\Omega_{\operatorname{cl}}^{1}(B)\oplus\mathcal{A}^{\gamma},
  \end{equation}
  and the first Poisson cohomology of $\Pi$ is of the form
  \begin{equation}\label{PCF3}
  H_{\Pi}^{1}(E) \cong H_{\operatorname{dR}}^{1}(B)\oplus\frac{\mathcal{A}^{\gamma}}{\operatorname{Ham}(E,P)}.
  \end{equation}
  Here, $\mathcal{A}^{\gamma}$ is a Lie subalgebra of vertical Poisson vector fields of $P$ defined in \eqref{ALP}.
\end{proposition}

\begin{proof}
If \eqref{PCF2} holds, then $H_{\bar{\partial}_{1,0}}^{2}=\{0\}$. Therefore, by definition \eqref{Hom}, we get
$\ker\rho^{\gamma}=\mathcal{A}^{\gamma}$. Hence, the decompositions \eqref{eq:PoissSplit} and \eqref{PC} coincide with \eqref{PCF9} and \eqref{PCF3}, respectively.
\end{proof}

As a consequence of Proposition \ref{prop:CasimCase0}, hypotheses \eqref{PC2} and \eqref{PCF2} imply that the properties of the first Poisson cohomology of $\Pi$ are controlled by the $\gamma$-dependent Lie algebra $\mathcal{A}^{\gamma}$. In fact, the algebra $\mathcal{A}^{\gamma}$ depends on an equivalence class of the Poisson connection $\gamma$ on $(E\overset{\pi}{\rightarrow}B,P)$. Indeed, suppose we have another Poisson connection $\widetilde{\gamma}$ on $(E\overset{\pi}{\rightarrow}B,P)$ which is equivalent to $\gamma$ in following sense: there exists a 1-form $\varrho\in\Omega_{B}^{1,0}(E)$ such that $\operatorname{hor}^{\widetilde{\gamma}}(u)=\operatorname{hor}^{\gamma}(u)+P^{\sharp}d\varrho(u)$, for every $u\in\mathfrak{X}(B)$. Then, $\widetilde{\gamma}\sim\gamma$ implies $\mathcal{A}^{\widetilde{\gamma}}=\mathcal{A}^{\gamma}$.

It is useful also to single out the Lie subalgebra of $\mathcal{A}^{\gamma}$ consisting of all vertical Poisson vector fields on $(E,P)$ that preserve the horizontal subbundle of $\gamma$,
\[
\mathcal{A}_{0}^{\gamma}:=\{Y\in\operatorname{Poiss}_{V}(E,P)\mid [\operatorname{hor}^{\gamma}(u),Y]=0\text{\ }\forall u\in\mathfrak{X}(B)\}.
\]
Then, taking into account \eqref{PC2}, we get
\[
\operatorname{Ham}_{0}(E,P) := \mathcal{A}_{0}^{\gamma}\cap\operatorname{Ham}(E,P)=\{P^{\sharp}dF\mid L_{\operatorname{hor}^{\gamma}(u)}F\in\pi^{\ast}C^{\infty}(B)~\forall
u\in\mathfrak{X}(B)\}.
\]

An interesting situation occurs when
\begin{equation}\label{N}
\mathcal{A}^{\gamma}\cong\frac{\mathcal{A}_{0}^{\gamma}}{\operatorname{Ham}_{0}(E,P)}\oplus\operatorname{Ham}(E,P).
\end{equation}
In this case,
\begin{equation}\label{HH}
H_{\Pi}^{1}(E)\cong H_{\operatorname{dR}}^{1}(B)\oplus\frac{\mathcal{A}_{0}^{\gamma}}{\operatorname{Ham}_{0}(E,P)}.
\end{equation}

\begin{lemma}
Condition \eqref{N} is equivalent to the following: every $Y\in\mathcal{A}^{\gamma}$ admits the decomposition
\begin{equation}\label{N2}
Y=P^{\sharp}dG+Y_{0},
\end{equation}
where $Y_{0}\in\mathcal{A}_{0}^{\gamma}$ and $G\in C^{\infty}(E)$. Furthermore, given $\beta\in\Omega^{1}(B)\otimes_{C^{\infty}(B)} C^{\infty}(E)$ in \eqref{3F}, there exists $c\in\Omega^{1}(B)$ such that
\begin{equation}\label{N3}
\partial^{\gamma}_{1,0}G = \beta - c\otimes 1.
\end{equation}
\end{lemma}

Next, let us consider the following particular case. Let $E=B\times K$ be the product of a manifold $B$ equipped with zero Poisson structure and a Poisson manifold $(K,\Upsilon)$. Let $P$ be the product Poisson structure on $E$. Then, we have the trivial Poisson bundle $(E=B\times K,P)$ over $B$ with projection $\pi_{B}=\operatorname{pr}_{1}$ and the typical fiber $(K,\Upsilon)$. Consider the trivial Poisson connection $\gamma^{0}$ on $E$ associated with the canonical horizontal distribution $\ker(d\pi_{K})$, where $\pi_{K}=\operatorname{pr}_{2}$.

\begin{proposition}\label{prop:CasimCase}
Let $\Pi$ be a compatible coupling Poisson tensor on the trivial Poisson bundle $(E=B\times K,P)$ in the sense that $\Pi_{0,2}=P$ and the associated Poisson connection $\gamma$ is equivalent to the trivial one, $\gamma\sim\gamma^{0}$. Assume that $B$ is connected,
\begin{equation}\label{CC}
\operatorname{Casim}(K,\Upsilon)=\mathbb{R},
\end{equation}
condition \eqref{PCF2} holds, and $\mathcal{A}^{\gamma^{0}}$ admits splitting \eqref{N}. Then,
\begin{equation}\label{NN}
H_{\Pi}^{1}(E)\cong H_{\operatorname{dR}}^{1}(B)\oplus H_{\Upsilon}^{1}(K).
\end{equation}
\end{proposition}

\begin{proof}
By the connectedness of $B$, it is easy to see that $\mathcal{A}_{0}^{\gamma^{0}}\cong\operatorname{Poiss}(K,\Upsilon)$, where the isomorphism is given by the horizontal lift in $E\overset{\pi_{K}}{\rightarrow}K$ with horizontal distribution $\ker d\pi_{B}$. Moreover, we claim that $\operatorname{Ham}_{0}(E,P)$ $\cong\operatorname{Ham}(K,\Upsilon)$. Indeed, pick a $Y=P^{\sharp}dF\in\operatorname{Ham}_{0}(E,P)$. Since hypothesis \eqref{CC} implies \eqref{PC2}, we conclude that $L_{\operatorname{hor}^{\gamma^{0}}(u)}F\in\pi_{B}^{\ast}C^{\infty}(B)$ for any $u\in\mathfrak{X}(B)$. Then, fixing $y^{0}\in K$, we see that the function $\tilde{F}\in C^{\infty}(E)$ given by $\tilde{F}(x,y)=F(x,y)-F(x,y^{0})$ for $x\in B,y\in K$ is of the form $\tilde{F}=\pi_{K}^{\ast}f$ for a certain $f\in C^{\infty}(K)$. Consequently, $Y$ is $\pi_{N}$-related with the Hamiltonian vector field $\Upsilon^{\sharp}df$.
\end{proof}

\begin{example}
Consider the trivial Poisson bundle $(\pi=\operatorname{pr}_{1}:E=B\times\mathbb{R}^{3},P)$, where the base $B=\mathbb{R}^{1}\times\mathbb{S}^{1}=\{(t,\varphi\operatorname{mod}2\pi)\}$ is the 2-cylinder and the typical fiber $(\mathbb{R}^{3},\Lambda)$ is given by the Lie-Poisson structure \eqref{LA2}. We already know that in this case, the projectability condition \eqref{PC2} holds. It is clear that \eqref{PCF2} is also satisfied. By analyzing equations \eqref{3F}, \eqref{N2}, \eqref{N3}, one can show that decomposition \eqref{N} is true for $\mathcal{A}^{\gamma^{0}}$. In this case, an arbitrary compatible coupling Poisson structure on the trivial Poisson bundle $E$ such that $\gamma\sim\gamma^{0}$ has the form
\begin{align*}
\Pi_{\varrho}  &  =\frac{1}{2\left(1-\Delta_{\varrho}\right)}\left(\frac{\partial}{\partial t} + (\psi\times\frac{\partial\varrho^{(1)}}{\partial x})\cdot\frac{\partial}{\partial x}\right) \wedge \left(\frac{\partial}{\partial\varphi}+(\psi\times\frac{\partial\varrho^{(2)}}{\partial x})\cdot\frac{\partial}{\partial x}\right) \\
&  +\frac{\partial}{\partial x_{1}} \wedge \left(x_{2}\frac{\partial}{\partial x_{2}} + x_{3}\frac{\partial}{\partial x_{3}}\right),
\end{align*}
where $\psi=(0,-x_{3},x_{2})$, $\varrho=\varrho^{(1)}(t,\varphi,x)dt+\varrho^{(2)}(t,\varphi,x)d\varphi$ is an arbitrary horizontal 1-form on $E$ and
\[
\Delta_{\varrho}:=\frac{\partial\varrho^{(2)}}{\partial t}-\frac{\partial\varrho^{(1)}}{\partial\varphi}+\psi(x)\cdot\left(\frac {\partial\varrho^{(1)}}{\partial x} \times \frac{\partial\varrho^{(2)}}{\partial x}\right)
\]
is the Hamiltonian of the curvature of $\gamma$ \cite{Vo-05}. Applying Proposition \ref{prop:CasimCase} to $\Pi_{\varrho}$ and taking into account Example \ref{ex:OpenBook}, we get
\[
H_{\Pi_{\varrho}}^{1}(E)\cong\mathbb{R}\oplus\mathbb{R}^{4}.
\]
\end{example}

\section{Acknowledgements}

The authors are very grateful to Jos\'e A. Vallejo for fruitful discussions and to anonymous Referees for critical comments and useful observations. The research was partially supported by CONACYT under the grant no. 219631.

\bibliographystyle{acm}
\bibliography{OnTheSplitting(ArXiv).bbl}
\end{document}